\pdfoutput=1
\documentclass[10pt,twoside,twocolumn]{IEEEtran}

\usepackage{ifthen}
\usepackage[pdftex]{graphicx} 

\usepackage{booktabs} 
\usepackage{cite} 
\usepackage[cmex10]{amsmath}
\usepackage{amssymb} 
\usepackage{amsfonts}

\usepackage{fixltx2e}

\usepackage{algorithm} \usepackage{algorithmic}

 \usepackage{subfigure}

\newcommand{\eg}{{\it e.g.}}  \newcommand{\ie}{{\it i.e.}}

 \newcommand{\reals}{{\mathbb R}}
 
\newcommand\plusvee{\mathrel{\ooalign{\lower.5ex
\hbox{$\scriptstyle\vee\mkern.5mu$}\cr\hidewidth\raise.450ex
\hbox{$\scriptstyle+$}\cr}}}

\newcommand{\argmin}{\operatornamewithlimits{argmin}}

\newcommand{\minimize}{\operatornamewithlimits{minimize}}
\newcommand{\maximize}{\operatornamewithlimits{maximize}}

\newboolean{draft}

\newcommand{\isdraft}[2]{\ifthenelse{\boolean{draft}}{#1}{#2}}

\newtheorem{theorem}{Theorem} 
\newtheorem{lemma}[theorem]{Lemma}
\newtheorem{proposition}[theorem]{Proposition}

\title{DCOOL-NET: Distributed cooperative localization for sensor
networks}

\author{Cl\'audia Soares*,~\IEEEmembership{Student Member,~IEEE,}
  Jo\~ao Xavier,~\IEEEmembership{Member,~IEEE,} and Jo\~ao
  Gomes,~\IEEEmembership{Member,~IEEE} \thanks{This work was partially
    supported by FCT, under projects PEst-OE/EEI/LA0009/2011,
    CMU-PT/SIA/0026/2009, PTDC/EEA-CRO/104243/2008, and Grant
    SFRH/BD/72521/2010. The authors are with the Institute for Systems
    and Robotics (ISR), Instituto Superior T\'ecnico, Technical
    University of Lisbon, 1049-001 Lisboa, Portugal (e-mail:
    csoares@isr.ist.utl.pt; jxavier@isr.ist.utl.pt;
    jpg@isr.ist.utl.pt).}}

\begin{document} 

\maketitle

\begin{abstract}
  We present DCOOL-NET, a scalable distributed in-network algorithm
  for sensor network localization based on noisy range measurements.
  DCOOL-NET operates by parallel, collaborative message passing
  between single-hop neighbor sensors, and involves simple
  computations at each node. It stems from an application of the
  majorization-minimization (MM) framework to the nonconvex
  optimization problem at hand, and capitalizes on a novel convex
  majorizer.  The proposed majorizer is endowed with several desirable
  properties and represents a key contribution of this work. It is a
  more accurate match to the underlying nonconvex cost function than
  popular MM quadratic majorizers, and is readily amenable to
  distributed minimization via the alternating direction method of
  multipliers (ADMM).  Moreover, it allows for low-complexity, fast
  Nesterov gradient methods to tackle the ADMM subproblems induced at
  each node.  Computer simulations show that DCOOL-NET achieves
  comparable or better sensor position accuracies than a state-of-art
  method which, furthermore, is not parallel.
\end{abstract}


\begin{center} \bfseries EDICS Category: SEN-DIST SEN-COLB \end{center}
\IEEEpeerreviewmaketitle

\section{Introduction}
\label{sec:introduction}

Applications of sensor networks in environmental and infrastructure
monitoring, surveillance, or healthcare (\eg, body area networks,
infrastructure area networks, as well as staff and equipment
localization~\cite{lai2011}), typically rely on known sensor nodes'
positions, even if the main goal of the network is not localization. A
sensor network usually comprises a large set of miniature, low cost, low
power autonomous sensor nodes. In this scenario it is generally
unsuitable or even impossible to accurately deploy all sensor nodes in
a predefined location within the network monitored area. GPS is also
discarded as an option for indoor applications or due to cost and
energy consumption constraints.

\subsubsection*{Centralized methods}
\label{sec:related-work}


There is a considerable body of work on centralized sensor network
localization. These algorithms entail a central or fusion processing
node, to which all sensor nodes communicate their range
measurements. The energy spent on communications can degrade the
network operation lifetime, since centralized architectures are prone
to data traffic bottlenecks close to the central node. Resilience to
failure, security and privacy issues are, also, not naturally
accounted for by the centralized architecture. Moreover, as the number
of nodes in the network grows, the problem to be solved at the central
node becomes increasingly complex, thus raising scalability
concerns. Focusing on recent work, several different approaches are
available, such as~\cite{KellerGur2011}, where sensor network
localization is formulated as a regression problem over adaptive
bases, or~\cite{DestinoAbreu2011}, where the strategy is to perform
successive minimizations of a weighted least squares cost function
convolved with a Gaussian kernel of decreasing variance. Another
successfully pursued approach is to perform semi-definite (SDP) or
second order cone (SOCP) relaxations to the original non-convex
problem~\cite{OguzGomesXavierOliveira2011,BiswasLiangTohYeWang2006}.
In~\cite{OguzGomesXavierOliveira2011} and~\cite{KorkmazVeen2009} the
majorization-minimization (MM) framework was used with quadratic cost
functions to derive centralized approaches to the sensor network
localization problem.
Finally, another widely used methodology relies on multidimensional scaling (MDS),
where the sensor network localization problem is posed as a
least-squares problem, as
in~\cite{ShangRumiZhangFromherz2004,CostaPatwariHero2006}. 


\subsubsection*{Distributed methods}
\label{sec:related-work-d}

Distributed approaches for sensor network localization, requiring no
central or fusion node, have been less frequent, despite the more
suited nature of this computational paradigm to the problem at hand
and the discussed advantages in applications.  
The work in~\cite{SrirangarajanTewfikLuo2008} proposes a parallel
distributed algorithm. However, the sensor network localization
problem 
adopts
a discrepancy function between squared distances which, unlike the
ones in maximum likelihood (ML) methods, is known to greatly amplify
measurement errors and outliers.  The convergence properties of the
algorithm are not studied theoretically.  The work
in~\cite{ChanSo2009} also considers network localization outside a ML
framework. The approach proposed in~\cite{ChanSo2009} is not parallel,
operating sequentially through layers of nodes: neighbors of anchors
estimate their positions and become anchors themselves, making it
possible in turn for their neighbors to estimate their positions, and
so on. Position estimation is based on planar geometry-based
heuristics.  In~\cite{KhanKarMoura2010}, the authors propose an
algorithm with assured asymptotic convergence, but
the solution is computationally complex since a triangulation set must
be calculated, and matrix operations are pervasive. Furthermore, in
order to attain good accuracy, a large number of range measurement
rounds must be acquired, one per iteration of the algorithm, thus
increasing energy expenditure. On the other hand, the algorithm
presented in~\cite{ShiHeChenJiang2010} and based on the non-linear
Gauss Seidel framework, has a pleasingly simple implementation,
combined with the convergence guarantees inherited from the non-linear
Gauss Seidel framework. Notwithstanding, this algorithm is sequential,
\ie, nodes perform their calculations in turn, not in a parallel
fashion. This entails the existence of a network-wide coordination
procedure to precompute the processing schedule upon startup, or
whenever a node joins or leaves the network.

\subsubsection*{Contribution}
\label{sec:contributions} In the present work, we set forth a
distributed algorithm, termed DCOOL-NET, for sensor network
localization, in which all nodes work in parallel and collaborate only
with single-hop neighbors.  DCOOL-NET is obtained by casting the
localization problem in a principled ML framework and following a
majorization-minorization approach~\cite{HunterLange2004} to
tackle the resulting nonconvex optimization problem.

MM paradigm operates on a novel convex majorization function which we
tailored to match the particular nonconvex structure in the cost
function. The proposed majorizer is a much tighter approximation than
the popular MM quadratic majorizers, and enjoys several
optimization-friendly properties.  It is readily amenable to
distributed minimization via the alternating direction method of
multipliers (ADMM), with convergence guarantees.  Moreover, by
capitalizing on its special properties, we show how to setup
low-complexity algorithms based on the fast-gradient Nesterov methods
to address the ADMM subproblems at each node.

As in other existing distributed iterative approaches, there is no
theoretical guarantee \textit{a priori} that DCOOL-NET will find the
global minimum of the nonconvex cost function for any given
initialization (even in the more favorable centralized setting, we are
not aware of theoretical proofs establishing global convergence of
existing methods for
tackling~\eqref{eq:snlOptimizationProblem}). However, computer
simulations show that our approach yields comparable or better sensor
position accuracies, for the same initialization, than the
state-of-art method in~\cite{ShiHeChenJiang2010} which, like
DCOOL-NET, is easily implementable and ensures descent of the cost
function at each iteration, but, contrary to DCOOL-NET, it is not a
parallel method.

\section{Problem statement}
\label{sec:problem-statement} The sensor network is represented as an
undirected connected graph $\mathcal{G} =
(\mathcal{V},\mathcal{E})$. The node set $\mathcal{V} = \{1,2, \dots,
n\}$ denotes the sensors with unknown positions. There is an edge $i
\sim j \in {\mathcal E}$ between sensors $i$ and $j$ if and only if a
noisy range measurement between nodes $i$ and $j$ is available at both
of them and nodes $i$ and $j$ can communicate with each other.  The
set of sensors with known positions, hereafter called anchors, is
denoted by ${\mathcal A} = \{ 1, \ldots, m \}$. For each sensor $i \in
{\mathcal V}$, we let ${\mathcal A}_i \subset {\mathcal A}$ be the
subset of anchors (if any) relative to which node $i$ also possesses a
noisy range measurement.

Let $\reals^p$ be the space of interest ($p=2$ for planar networks,
and $p=3$ otherwise).  We denote by $x_i \in \reals^p$ the position of
sensor $i$, and by $d_{ij}$ the noisy range measurement between
sensors $i$ and $j$, available at both $i$ and
$j$. Following~\cite{ShiHeChenJiang2010}, we assume $d_{ij} = d_{ji}$\footnote{This entails no loss of generality: it is readily seen that, if $d_{ij} \neq d_{ji}$, then it suffices to replace $d_{ij}  \leftarrow (d_{ij} + d_{ji})/2$ and $d_{ji} \leftarrow (d_{ij} + d_{ji})/2$ in the forthcoming optimization problem~\eqref{eq:snlOptimizationProblem}.}. Anchor positions
are denoted by $a_{k} \in \reals^{p}$. We let $r_{ik}$ denote the
noisy range measurement between sensor $i$ and anchor $k$, available
at sensor $i$.

The distributed network localization problem addressed in this work
consists in estimating the sensors' positions $x = \{ x_i\, : \, i \in
\mathcal{V} \}$, from the available measurements $\{ d_{ij} \, : \, i
\sim j \} \cup \{ r_{ik} \, : \, i \in {\mathcal V}, k \in {\mathcal
A}_i \}$, through collaborative message passing between neighboring
sensors in the communication graph~${\mathcal G}$.

Under the assumption of zero-mean, independent and
identically-distributed, additive Gaussian measurement noise, the
maximum likelihood estimator for the sensor positions is the solution
of the optimization problem
\begin{equation}
  \label{eq:snlOptimizationProblem} 
  \minimize_{x} f(x),
\end{equation} 
where
\begin{equation*} 
  f(x) = \sum _{i \sim j} (\|x_{i} - x_{j}\| - d_{ij})^2 + \sum_{i}
  \sum_{k \in \mathcal{A}_{i}} (\|x_{i}-a_{k}\| - r_{ik})^2.
\end{equation*} 
Problem~\eqref{eq:snlOptimizationProblem} is non-convex and difficult
to solve.  Even in the centralized setting (\ie, all measurements are
available at a central node) currently available iterative techniques
don't claim convergence to the global optimum. Also, even with
noiseless measurements, multiple solutions might exist due to
ambiguities in the network topology itself. For noiseless unambiguous
topologies, also called \emph{localizable} networks, the measurement
noise may create ambiguities. However, recent studies
like~\cite{AndersonShamesMaoFidan2010} provide theoretical guarantees
to confidently consider networks which are approximately localizable,
\ie, whose localizability resists to a reasonable amount of
measurement noise.

\section{DCOOL-NET}
\label{sec:conv-major-funct}

We propose a distributed algorithm, termed DCOOL-NET, to tackle
problem~\eqref{eq:snlOptimizationProblem}.  Starting from an
initialization $x[0]$ for the unknown sensors' positions $x$,
DCOOL-NET generates a sequence of iterates $\left( x[l] \right)_{l
\geq 1}$ which, hopefully, converges to a solution
of~\eqref{eq:snlOptimizationProblem}.  Conceptually, DCOOL-NET is
obtained by applying the majorization minimization (MM)
framework\cite{HunterLange2004} to~\eqref{eq:snlOptimizationProblem}:
at each iteration~$l$, a majorizer of $f$, tight at the current
iterate $x[l]$, is minimized to obtain the next iterate $x[l+1]$.  The
algorithm is outlined in Algorithm~\ref{alg:mm} for a fixed number of
iterations~$L$.
\begin{algorithm}
  \caption{DCOOL-NET}
  \label{alg:mm}
  \begin{algorithmic}[1] 
    \REQUIRE $x[0]$ 
    \ENSURE $x[L]$ 
    \FOR{$l=0$ \TO $L-1$} 
    \STATE $x[l+1]= \argmin_{x} F(x \, | \,
    x[l])$ \label{alg:mm-min}
    \ENDFOR 
    \RETURN $x[L]$
  \end{algorithmic}
\end{algorithm} 
Here, $F( \cdot\, | \, x[l])$ denotes a majorizer of
$f$ (\ie, $f(x) \leq F( x \, | \, x[l] )$ for all $x$) which is tight
at $x[l]$ (\ie, $f( x[l] ) = F( x[l]\, |\, x[l] )$).  The majorizer is
detailed in the next section. Note that $f( x[l+1] ) \leq f( x[l] )$,
that is, $f$ is monotonically decreasing along iterations, an
important property of the MM framework.

DCOOL-NET is a distributed algorithm because, as we shall see, the
minimization of the upper-bounds $F$ can be achieved in a distributed
manner.

\section{Majorization function}
\label{sec:major-funct}

Commonly, MM techniques resort to quadratic majorizers which, albeit
easy to minimize, show a considerable mismatch with most cost
functions (in particular, with $f$
in~\eqref{eq:snlOptimizationProblem}). To overcome this problem, we
introduce a key novel majorizer. It is specifically adapted to~$f$,
tighter than a quadratic, convex, and easily optimizable.

Before proceeding it is useful to
rewrite~\eqref{eq:snlOptimizationProblem} as
\begin{equation*} 
  f(x) = \sum _{i \sim j} f_{ij}(x_{i},x_{j}) +
  \sum_{i} \sum_{k \in \mathcal{A}_{i}} f_{ik}(x_{i}),
\end{equation*} 
where $f_{ij}(x_{i},x_{j}) = \phi_{d_{ij}}(x_{i}-x_{j})$ and
$f_{ik}(x_{i}) = \phi_{r_{ik}}(x_{i}-a_{k})$, both defined in terms of
the basic building block
\begin{equation}
  \label{eq:original-cost-generic} 
  \phi_{d}(u) = (\|u\| - d)^{2}.
\end{equation}

\subsection{Majorization function
for~\eqref{eq:original-cost-generic}}
\label{sec:major-funct-one}

Let $v \in {\mathbb R}^p$ be given, assumed nonzero. We provide a
majorizer $\Phi_d( \cdot\, |\, v )$ for $\phi_d$ in~\eqref{eq:original-cost-generic} which is tight at
$v$, i.e., $\phi_d (u) \leq \Phi_d( u\, | \, v )$ for all $u$ and
$\phi_d( v ) = \Phi_d( v \, |\, v )$.

\begin{proposition}
  \label{prop:majorization-properties} 
  Let
  \begin{equation}
  \label{eq:majorization-simple} 
  \Phi_{d}(u | v) = \max\left\{g_{d}(u), h_{d}(v^{\top} u / \left\| v
    \right\|-d) \right\},
\end{equation} 
where
\begin{equation}
  \label{eq:squarepos} 
  g_{d}(u) = \left(\|u\| - d\right)_{+}^{2},
\end{equation} 
$(r)_+^2 = \left(\max\{0,r\}\right)^2$, and
\begin{equation}
  \label{eq:huber} 
  h_{R}(r) = \left \{
    \begin{array}{ll} 
      2 R | r | -R^2 & \text{if } | r | \geq R\\ 
      r^2 & \text{if } | r | < R,
    \end{array}
  \right.
\end{equation}
is the Huber function of parameter $R$.  Then, the function
$\Phi_{d}( \cdot\, |\, v)$ is convex, is tight at $v$, and majorizes
$\phi_d$.
\end{proposition}
\begin{proof} 
  See Appendix~\ref{sec:proof-theor-thcvxmaj}.
\end{proof}

The tightness of the proposed majorization function is illustrated in
Fig.~\ref{fig:cvxapprox}, in which we depict, for an one-dimensional
argument~$u$, $d = 0.5$ and $v = 0.1$: the nonconvex cost function
in~\eqref{eq:original-cost-generic}, the proposed majorizer
in~\eqref{eq:majorization-simple} and a quadratic majorizer
\begin{equation}
  \label{eq:quadratic-majorizer}
  Q_{d}( u | v ) = \|u\|^{2} + d^{2} - 2d \frac{v^{\top}u}{\|v\|},
\end{equation}
obtained through routine manipulations
of~\eqref{eq:original-cost-generic}, e.g., expanding the square and
linearizing $\left\| u \right\|$ at $v$, which is common in MM
approaches (c.f. \cite{OguzGomesXavierOliveira2011,KorkmazVeen2009}
for quadratic majorizers applied to the sensor network localization
problem and~\cite{ForeroGiannakis2012} for an application in robust
MDS).  Clearly, the proposed convex majorizer is a better
approximation to the nonconvex cost function\footnote{The fact that
  both majorizers have coincident minimum is an artifact of this toy
  example, and does not hold in general.}.  As an expected corollary,
it also outperforms in accuracy the quadratic majorizer when embedded
in the MM framework, as shown in the experimental results of
Sec.~\ref{sec:experimental-results}.

\begin{figure}[!t] 
  \centering
  \includegraphics[width=0.49\textwidth]{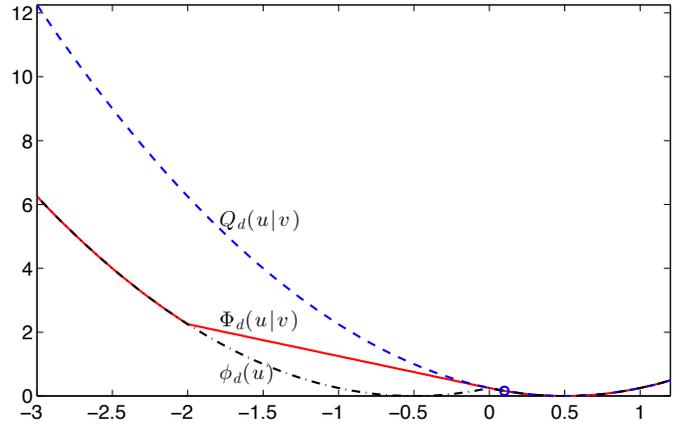}
  \caption{Nonconvex cost function (black, dash point)
    in~\eqref{eq:original-cost-generic} against the proposed majorizer
    (red, solid) in~\eqref{eq:majorization-simple} and a vanilla quadratic
    majorizer (blue, dashed) in~\eqref{eq:quadratic-majorizer}, for $d =
    0.5$ and $v = 0.1$. The proposed convex majorizer is a much more accurate
    approximation.}
  \label{fig:cvxapprox}
\end{figure}

\subsection{Majorization function for~\eqref{eq:snlOptimizationProblem}}
\label{sec:major-funct-sum}

Now, for given $x[l]$, consider the function
\begin{equation}
  \label{eq:cvxapproximationFull} 
  F( x \, | \, x[l] ) = \sum _{i \sim j} F_{ij}(x_i, x_j) + \sum_{i}
  \sum_{k \in \mathcal{A}_{i}} F_{ik}(x_i),
\end{equation} 
where
\begin{IEEEeqnarray}{rCl}
  \label{eq:maj-functions-expr} 
  F_{ij}(x_i, x_j) &=& \Phi_{d_{ij}}(x_{i}-x_{j}\, |\, x_{i}[l] - x_{j}[l])
\end{IEEEeqnarray} 
and
\begin{IEEEeqnarray}{rCl}
  \label{eq:maj-functions-s2a} 
  F_{ik}(x_{i}) &=& \Phi_{r_{ik}}(x_{i}-a_{k}\, |\, x_{i}[l] - a_{k}).
\end{IEEEeqnarray} 
Given Proposition~\ref{prop:majorization-properties}, it is clear that
it majorizes $f$ and is tight at $x[l]$. Moreover, it is convex as a
sum of convex functions.

\section{Distributed minimization}
\label{sec:distr-algor} 

At the $l$-th iteration of DCOOL-NET, the convex function
in~\eqref{eq:cvxapproximationFull} must be minimized.  We now show how
this optimization problem can be solved collaboratively by the network
in a distributed, parallel manner.

\subsection{Problem reformulation}
\label{sec:probl-reform}

In the distributed algorithm, the working node will operate on local
copies of the estimated positions of its neighbors and of itself.
So, it is convenient to introduce new
variables.  Let ${\mathcal V}_i = \{ j\, : \, j \sim i \}$ denote the
neighbors of sensor~$i$. We also define the closed neighborhood
$\overline{\mathcal V}_i = {\mathcal V}_i \cup \{ i \}$.  For each
$i$, we duplicate $x_i$ into new variables $y_{ji}$, $j \in
\overline{\mathcal V}_i$, and $z_{ik}$, $k \in {\mathcal A}_i$.
This choice of notation is not fortuitous: the first subscript reveals which physical node will store the variable, in our proposed implementation; thus, $x_i$ and $z_{ik}$ are stored at node~$i$, whereas $y_{ji}$ is managed by node~$j$.
 We
write the minimization of~\eqref{eq:cvxapproximationFull} as the
optimization problem
\begin{equation}
  \label{eq:admmproblem}
  \begin{array}{ll} 
    \minimize & F(y, z) \\ \mbox{subject to } & y_{ji}
    = x_i, \qquad j \in \overline{\mathcal{V}}_{i}\\ & z_{ik} = x_{i},
    \qquad k \in \mathcal{A}_{i},\\
  \end{array}
\end{equation} 
where $y = \{y_{ji} : i \in \mathcal{V}, j \in
\overline{\mathcal{V}}_{i}\}$, $z = \{z_{ik} : i \in \mathcal{V}, k
\in \mathcal{A}_{i}\}$, and 
\begin{IEEEeqnarray}{rC} 
  \isdraft{\label{eq:admmcost}}{\nonumber} 
  F(y, z) =& \sum_{i \sim j} \left(F_{ij}\left(y_{ii},y_{ij} \right) +
    F_{ij}\left(y_{ji},y_{jj} \right) \right) + \isdraft{}{\\
  \label{eq:admmcost} 
  &+} 2 \sum_{i} \sum_{k \in \mathcal{A}_{i}} F_{ik}\left(z_{ik}\right).
\end{IEEEeqnarray} 
In passing from~\eqref{eq:cvxapproximationFull}
to~\eqref{eq:admmproblem} we used the identity $F_{ij}(x_{i},x_{j})=
\frac{1}{2} F_{ij}\left(y_{ii}, y_{ij}\right) + \frac{1}{2}
F_{ij}\left(y_{ji}, y_{jj}\right)$, due to $y_{ji} = x_{i}$. Also, for
convenience, we have rescaled the objective by a factor of two.

\subsection{Algorithm derivation}
\label{sec:algor-deriv}

Problem~\eqref{eq:admmproblem} is in the form
\begin{equation}
  \label{eq:admmproblem-general}
  \begin{array}{ll} 
    \minimize & F(y,z) + G(x)\\ \mbox{subject to } & A(y,z) + Bx = 0
  \end{array}
\end{equation} 
where $F$ is the convex function in~\eqref{eq:admmcost}, $G$ is the
identically zero function, $A$ is the identity operator and $B$ is a
matrix whose rows belong to the set $\{-e_i^\top ,\: i \in
\mathcal{V}\}$, being $e_i$ the $i$th column of the identity matrix of
size $|\mathcal{V}|$. In the presence of a connected network $B$ is
full column rank, so the problem is suited for the Alternating
Direction Method of Multipliers (ADMM). See
\cite{BoydParikhChuPeleatoEckstein2011} and references therein for
more details on this method. See
also~\cite{SchizasRibeiroGiannakis2008,ZhuGiannakisCano2009,ForeroCanoGiannakis2010,BazerqueGiannakis2010,ErsegheZennaroAneseVangelista2011,MotaXavierAguiarPuschel}
for applications of ADMM in distributed optimization settings.

Let $\lambda_{ji}$ be the Lagrange multiplier associated with the
constraint $y_{ji} = x_i$ and $\lambda = \left \{\lambda_{ji} \right
\}$ the collection of all such multipliers. Similarly, let $\mu_{ik}$
be the Lagrange multiplier associated with the constraint $z_{ik} =
x_i$ and $\mu = \{ \mu_{ik} \}$.  The ADMM framework generates a
sequence $\left( y(t), z(t), x(t), \lambda(t), \mu(t) \right)_{t \geq
  1}$ such that
\begin{IEEEeqnarray}{lCr} 
  (y(t+1), z(t+1)) = \argmin_{y,z} \mathcal{L}_\rho\left( y, z, x(t), \lambda(t), \mu(t) \right)
  \isdraft{}{\nonumber \\ }
  \label{eq:admm-secondary-vars} \\ 
  x(t+1) = \argmin_x \mathcal{L}_\rho\left( y(t+1), z(t+1), x,
    \lambda(t), \mu(t) \right) 
  \isdraft{}{\nonumber \\ }
  \label{eq:admm-primary-vars} \\
  \label{eq:admm-dual-update} \lambda_{ji}(t+1) = \lambda_{ji}(t) +
  \rho (y_{ji}(t+1) - x_i(t+1))\\
  \label{eq:admm-dual-update2} 
  \mu_{ik}(t+1) = \mu_{ik}(t) + \rho (z_{ik}(t+1) - x_i(t+1)),
\end{IEEEeqnarray} 
where $\mathcal{L}_\rho$ is the augmented Lagrangian defined as
\begin{IEEEeqnarray}{lCr} 
  {\mathcal L}_\rho(y, z, x, \lambda, \mu) \isdraft{&=&}{=}
  F(y,z) + \sum_i \sum_{j \in \overline{\mathcal V}_i} \left (
    \rule{0ex}{2ex}\lambda_{ji}^\top (y_{ji} - x_i) +
    \isdraft{}{\right . \nonumber \\ \left .} \frac{\rho}{2} \left\|
      y_{ji} - x_i \right\|^2  \rule{0ex}{2ex} \right )+ 
  \isdraft{\nonumber \\ && }{}
  \sum_i \sum_{k \in {\mathcal A}_i} \left (   \mu_{ik}^\top
    (z_{ik} - x_i ) + \frac{\rho}{2} \left\| z_{ik} - x_i \right\|^2
     \right ). \isdraft{}{\nonumber \\} 
  \label{eq:auglagrangian}
\end{IEEEeqnarray}
Here, $\rho>0$ is a pre-chosen constant.

In our implementation, we let node $i$ store the variables $x_{i}$,
$y_{ij}$, $\lambda_{ij}$, $\lambda_{ji}$, for $j \in
\overline{\mathcal{V}}_{i}$ and $z_{ik}$, $\mu_{ik}$, for $k \in
\mathcal{A}_{i}$.  Note that a copy of $\lambda_{ij}$ is maintained at
both nodes~$i$ and~$j$ (this is to avoid extra communication steps).
For $t = 0$, we can set $\lambda(0)$ and $\mu(0)$ to a pre-chosen
constant (e.g., zero) at all nodes.  Also, we assume that, at the
beginning of the iterations (\ie, for $t = 0$) node~$i$ knows $x_j
(0)$ for $j \in {\mathcal V}_i$ (this can be accomplished, e.g., by
having each node $i$ communicating $x_i(0)$ to all its
neighbors). This property will be preserved for all $t \geq 1$ in our
algorithm, via communication steps.
  
We now show that the minimizations in~\eqref{eq:admm-secondary-vars}
and~\eqref{eq:admm-primary-vars} can be implemented in a distributed
manner and require low computational cost at each node.




\section{ADMM: Solving Problem~\eqref{eq:admm-secondary-vars}}
\label{sec:solv-probl-secondary-vars}

As shown in Appendix~\ref{sec:proof-B}, the augmented Lagrangian in~\eqref{eq:auglagrangian} can be written as
\begin{IEEEeqnarray}{rCl} 
  \isdraft{\label{eq:auglagrangian-v2} }{\nonumber}
  \mathcal{L}_\rho\left( y, z, x,\lambda, \mu \right) &=& \sum_{i}
  \left( \sum_{j \in \overline{\mathcal{V}}_{i}} {\mathcal L}_{ij}\left(y_{ii},y_{ij},
      x_{j},\lambda_{ij}\right) + \isdraft{}{\right.\\
  \label{eq:auglagrangian-v2} 
  &&\left . +}\sum_{k \in \mathcal{A}_{i}}
    {\mathcal L}_{ik}\left(z_{ik}, x_{i},\mu_{ik} \right) \right )
\end{IEEEeqnarray} 
where
\begin{IEEEeqnarray*}{rCl} 
  {\mathcal L}_{ij}\left(y_{ii},y_{ij},
    x_{j},\lambda_{ij}\right) &=& F_{ij}\left(y_{ii},y_{ij}\right) +
  \lambda_{ij}^\top\left(y_{ij} - x_j\right) +\isdraft{}{\\ 
  &&+}\frac{\rho}{2} \left\|y_{ij} - x_j\right\|^2
\end{IEEEeqnarray*} 
and
\begin{IEEEeqnarray*}{rCl} 
  {\mathcal L}_{ik}\left(z_{ik}, x_{i},\mu_{ik}\right) &=&
  2F_{ik}\left(z_{ik}\right) + \mu_{ik}^\top\left(z_{ik} - x_i\right) + \isdraft{}{\\ 
  &&+}\frac{\rho}{2} \left\|z_{ik} - x_i\right\|^2.
\end{IEEEeqnarray*} 
In~\eqref{eq:auglagrangian-v2} we let $F_{ii} \equiv 0$.  It is clear
from~\eqref{eq:auglagrangian-v2} that
Problem~\eqref{eq:admm-secondary-vars} decouples across sensors~$i \in
{\mathcal V}$, since we are optimizing only over $y$ and $z$. Further,
at each sensor $i$, it decouples into two types of subproblems: one
involving the variables $y_{ij}, j \in \overline {\mathcal V}_i$,
given by
\begin{IEEEeqnarray}{RC}
  \label{eq:secondary-s2s-nodei} 
  \minimize_{y_{ij}, \:j \in \overline{\mathcal{V}}_i } & \sum_{j \in
    \overline{\mathcal{V}}_{i}} {\mathcal L}_{ij}\left(y_{ii},y_{ij},
    x_{j},\lambda_{ij}\right),
\end{IEEEeqnarray} 
and into $| {\mathcal A}_i |$ subproblems of the form
\begin{IEEEeqnarray}{RC}
  \label{eq:secondary-s2a-nodei} 
  \minimize_{z_{ik}} & {\mathcal L}_{ik}\left(z_{ik}, x_{i},\mu_{ik}\right),
\end{IEEEeqnarray} 
involving the variable $z_{ik}$, $k \in \mathcal{A}_i$.  Note that
problems related with anchors are simpler, and, since there are
usually few anchors in a network, they do not occur frequently.

\subsection{Solving Problem~\eqref{eq:secondary-s2s-nodei}}
\label{sec:problem-nodes-in-s_i}

First, note that node~$i$ can indeed address
Problem~\eqref{eq:secondary-s2s-nodei} since all the data defining it
is available at node~$i$: it stores $\lambda_{ji}(t)$, $j \in
\overline {\mathcal V}_i$, and it knows $x_j(t)$ for all neighbors $j
\in {\mathcal V}_i$ (this holds trivially for $t= 0$ by construction,
and it is preserved by our approach, as shown ahead).

To alleviate notation we now suppress the indication of the working
node $i$, \ie, variable $y_{ij}$ is simply written as
$y_{j}$. Problem~\eqref{eq:secondary-s2s-nodei} can be written as
\begin{IEEEeqnarray}{rCl} 
  \isdraft{\label{eq:secondary-s2s-nodei-v2}}{\nonumber}
  \minimize_{y_j, \:j \in \overline{\mathcal{V}}_i } &\sum_{j \in
    \mathcal{V}_i}\left( F_{ij}(y_i, y_j) + \frac{\rho}{2} \|y_j - \gamma_{ij}\|^2\right) \isdraft{}{\\
    \label{eq:secondary-s2s-nodei-v2} 
&}+\frac{\rho}{2} \|y_i - \gamma_{ii}\|^2,
\end{IEEEeqnarray} 
where $\gamma_{ij} = x_j - \frac{\lambda_{ij}}{\rho}$.

We make the crucial observation that, for fixed $y_i$, the problem is
separable in the remaining variables $y_j$, $j \in
\mathcal{V}_i$. This motivates
writing~\eqref{eq:secondary-s2s-nodei-v2} as the master problem
\begin{equation}
  \label{eq:secondary-s2s-Gamma} 
  \minimize_{y_i} H(y_i) = \sum_{j \in
    \mathcal{V}_i} H_{ij}(y_i) + \frac{\rho}{2} \|y_i - \gamma_{ii}\|^2,
\end{equation} 
where
\begin{equation}
  \label{eq:xi-function} 
  H_{ij}(y_i) = \min_{y_j} F_{ij}(y_i, y_j) +
  \frac{\rho}{2} \|y_j - \gamma_{ij}\|^2.
\end{equation}

We now state important properties of $H_{ij}$.

\begin{proposition}
  \label{prop:xi} 
  Define $H_{ij}$ as in~\eqref{eq:xi-function}. Then:
  
  \begin{enumerate}
  \item \label{item:optim-probl-uniq} Optimization
    problem~\eqref{eq:xi-function} has a unique solution $y_j$ for
    any given $y_i$, henceforth denoted $y_{j}^{\star}(y_{i})$;
  \item \label{item:xi-cvx-diff-grad}Function $H_{ij}$ is convex and
    differentiable, with gradient
    \begin{equation}
      \label{eq:xi-gradient} 
      \nabla H_{ij}(y_{i}) = \rho \left(y_j^\star(y_i) - \gamma_{ij} \right);
    \end{equation}
  \item \label{item:grad-lipsch}The gradient of $H_{ij}$ is Lipschitz continuous  with
    parameter $\rho$, \ie,
    \begin{equation*} 
      \|\nabla H_{ij}(u) - \nabla H_{ij}(v)\| \leq \rho \| u - v\|
    \end{equation*} for all $u, v \in \reals^{p}$.
  \end{enumerate}
\end{proposition}

\begin{proof}
  \begin{enumerate}
  \item Recall from~\eqref{eq:maj-functions-expr} that $F_{ij}(y_i,
    y_j) = \Phi_d(y_i - y_j \, | \, v)$ where $d = d_{ij}$ and $v =
    x_{i}[l] - x_{j}[l]$.  We have $H_{ij}(y_i) = \Theta(y_i -
    \gamma_{ij})$ where
    \begin{equation}
      \label{eq:urruty-form} 
      \Theta(w) = \min_{u} \Phi_d(u \, | \, v ) + \frac{\rho}{2}
      \left\|u - w\right\|^2.
    \end{equation} 
    Moreover, $u^\star$ solves~\eqref{eq:urruty-form} if and only if
    $y_j^\star = y_i - u^\star$ solves~\eqref{eq:xi-function}. Now,
    the cost function in~\eqref{eq:urruty-form} is clearly continuous,
    coercive (\ie, it converges to $+\infty$ as $\left\| u \right\|
    \rightarrow +\infty$) and strictly convex, the two last properties
    arising from the quadratic term. Thus, it has an unique solution;
  \item The function $\Theta$ is the Moreau-Yosida regularization of
    the convex function $\Phi_d( \cdot |
    v)$~\cite[XI.3.4.4]{UrrutyMarechal1993}. As $\Theta$ is known to
    be convex and $H_{ij}$ is the composition of $\Theta$ with an
    affine map, $H_{ij}$ is convex.  It is also known that the
    gradient of $\Theta$ is
    \begin{equation*} 
      \nabla \Theta(w) = \rho(w - u^\star(w))
    \end{equation*} 
    where $u^\star(w)$ is the unique solution
    of~\eqref{eq:urruty-form} for a given~$w$.  Thus,
    \begin{eqnarray*}
      \label{eq:gradFj} 
      \nabla H_{ij}(y_{i}) & = & \nabla \Theta( y_i - \gamma_{ij} ) \\
      & = & \rho ( y_i - \gamma_{ij} - u^\star( y_i - \gamma_{ij} ) ).
    \end{eqnarray*} 
    Unwinding the change of variable, \ie, using $y_j^\star(y_i) = y_i
    - u^\star(y_i - \gamma_{ij})$, we obtain~\eqref{eq:xi-gradient};
  \item Follows from the well known fact that the gradient of $\Theta$
    is Lipschitz continuous with parameter $\rho$.
  \end{enumerate}
\end{proof}

As a consequence, we obtain several nice properties of the
function~$H$.

\begin{theorem}
  \label{th:gamma-properties} 
  Function~$H$ in~\eqref{eq:secondary-s2s-Gamma} is strongly convex
  with parameter~$\rho$, \ie, $H - \frac{\rho}{2} \left\| \cdot
  \right\|^2$ is convex.  Furthermore, it is differentiable with
  gradient
  \begin{equation}
    \label{eq:gradGamma} 
    \nabla H (y_i) = \rho\sum_{j \in \mathcal{V}_i}
    \left(y_j^\star(y_i) - \gamma_{ij} \right) + \rho (y_i - \gamma _{ii}).
  \end{equation} 
  The gradient of $H$ is Lipschitz continuous with parameter $L_H = \rho
  (|\mathcal{V}_i|+1)$.
 \end{theorem}

\begin{proof} 
  Since $H$ is a sum of convex functions, it is convex. It is strongly
  convex with parameter $\rho$ due to the presence of the strongly
  convex term $\frac{\rho}{2} \left\| y_i - \gamma_{ii} \right\|^2$.
  As a sum of differentiable functions, it is differentiable and the
  given formula for the gradient follows from
  proposition~\ref{prop:xi}.  Finally, since $H$ is the sum of $|
  {\mathcal V}_i | +1$ functions with Lipschitz continuous gradient with
  parameter $\rho$ , the claim is proved.
\end{proof}

The properties established in Theorem~\ref{th:gamma-properties} show
that the optimization problem~\eqref{eq:secondary-s2s-Gamma} is suited
for Nesterov's optimal method for the minimization of strongly convex
functions with Lipschitz continuous gradient~\cite[Theorem~2.2.3]{Nesterov2004}.
The resulting algorithm is outlined in Algorithm~\ref{alg:nesterov},
which is guaranteed to converge to the solution
of~\eqref{eq:secondary-s2s-Gamma}.
\begin{algorithm}
  \caption{Nesterov's optimal method
for~\eqref{eq:secondary-s2s-Gamma} }
  \label{alg:nesterov}
  \begin{algorithmic}[1] 
    \STATE $\hat y_i (0) = y_i (0)$ 
    \FOR{$s \geq 0$} 
    \STATE $y_i (s+1) = \hat y_i (s) - \frac{1}{L_H} \nabla H(y_i
    (s))$ 
    \STATE $\hat y_i (s+1) = y_i (s+1) + \frac{\sqrt{L_H} -
      \sqrt{\rho}}{\sqrt{L_H} + \sqrt{\rho}}(y_i (s+1) - y_i (s))$
    \ENDFOR
  \end{algorithmic}
\end{algorithm}

\subsection{Solving problem~\eqref{eq:xi-function}}
\label{sec:problem-xi-function}

It remains to show how to solve~\eqref{eq:xi-function} at a given
sensor node. Any off-the-shelf convex solver, e.g. based on
interior-point methods, could handle it. However, we present a simpler
method that avoids expensive matrix operations, typical of interior
point methods, by taking advantage of the problem structure at
hand. This is important in sensor networks where the sensors have
stringent computational resources.

 First, as shown in the proof of Proposition~\ref{prop:xi}, it suffices
to focus on solving~\eqref{eq:urruty-form} for a given~$w$:
solving~\eqref{eq:xi-function} amounts to solving~\eqref{eq:urruty-form}
for $w = y_i - \gamma_{ij}$ to obtain $u^\star = u^\star(w)$ and set
$y_j^\star(y_i) = y_i - u^\star$.

Note from~\eqref{eq:majorization-simple} that
$\Phi_d( \cdot | v)$ only depends on $v / \left\| v \right\|$, so we
can assume, without loss of generality, that $\left\| v \right\| = 1$.

From~{\eqref{eq:majorization-simple}, we see that
Problem~\eqref{eq:urruty-form} can be rewritten as
\begin{equation}
\label{step1}
\begin{array}{ll} 
  \minimize & r + \frac{\rho}{2} \left\|
    u - w \right\|^2 \\ 
  \mbox{subject to } & g_d( u ) \leq r \\ & h_d( v^\top u - d ) \leq r,
\end{array}
\end{equation} 
with optimization variable $(u, r)$. The Lagrange dual (c.f., for
example,~\cite{UrrutyMarechal1993}) of~\eqref{step1} is given by
\begin{equation}
\label{step2}
\begin{array}{ll} 
  \maximize & \psi(\omega) \\ 
  \mbox{subject to } & 0 \leq \omega \leq 1,
\end{array}
\end{equation} 
where $\psi(\omega) = \inf \{ \Psi( \omega, u )\, :\, u \in {\mathbb
  R}^n \}$ and
\begin{equation}
  \label{defvarphi} 
  \Psi(\omega, u) = \frac{\rho}{2} \left\| u - w
  \right\|^2 + \omega g_d ( u ) + ( 1 - \omega ) h_d( v^\top u - d).
\end{equation}

We propose to solve the dual problem~\eqref{step2}, which involves the
single variable~$\omega$, by bisection: we maintain an interval $[ a ,
b ] \subset [ 0 , 1 ]$ (initially, they coincide); we evaluate $\dot
\psi( c)$ at the midpoint $c = (a+b)/2$; if $\dot \psi(c) > 0$, we set
$a = c$, otherwise, $b = c$; the scheme is repeated until the
uncertainty interval is sufficiently small.

In order to make this approach work, we must prove first that the dual
function $\psi$ is indeed differentiable in the open interval $\Omega
= (0,1)$ and find a convenient formula for its derivative.  We will
need the following useful result from convex analysis.
\begin{lemma}
  \label{lem:diff-of-Psii} 
  Let $X \subset {\mathbb R}^n$ be an open convex set and $Y \subset
  \reals^{p}$ be a compact set. Let $F : X \times Y \rightarrow
  \reals$. Assume that $F( x, \cdot)$ is lower semi-continuous for all
  $x \in X$ and $F( \cdot, y )$ is concave and differentiable for all
  $y \in Y$. Let $f \, : \, X \rightarrow {\mathbb R}$, $f(x) = \inf
  \{ F( x, y)\, : \, y \in Y \}$. Assume that, for any $x \in X$, the
  infimum is attained at an unique $y^\star(x) \in Y$. Then, $f$ is
  differentiable everywhere and its gradient at $x \in X$ is given by
  \begin{equation}
    \label{eq:grad-of-psi-tilde} 
    \nabla f( x ) = \nabla F( x , y^{\star}(x))
  \end{equation} 
  where $\nabla$ refers to differentiation with respect to $x$.
\end{lemma}
  
\begin{proof} 
  This is essentially~\cite[VI.4.4.5]{UrrutyMarechal1993}, after one
  changes concave for convex, lower semi-continuous for upper
  semi-continuous and $\inf$ for $\sup$.
\end{proof}

Now, view $\Psi$ in~\eqref{defvarphi} as defined in $\Omega \times
{\mathbb R}^n$. Is is clear that $\Psi( \omega, \cdot )$ is lower
semi-continuous for all $\omega$ (in fact, continuous) and $\Psi(
\cdot, u )$ is concave (in fact, affine) and differentiable for all
$u$.  In fact, some even nicer properties hold.

\begin{lemma}
  \label{step3} 
  Let $\omega \in \Omega$.  The function $\Psi_\omega = \Psi( \omega,
  \cdot )$ is strongly convex with parameter $\rho$ and differentiable
  everywhere with gradient
  \begin{equation}
    \label{eq:ug}
    \nabla \Psi_\omega(u) = \rho (u - w) + 2 \omega ( u - \pi(u) ) +
    (1-\omega) \dot h_d( v^\top u - d ) v,
  \end{equation}
  where $\pi(u)$ denotes the projection of $u$ onto the closed ball of
  radius $d$ centered at the origin. Furthermore, the gradient of
  $\Psi_\omega$ is Lipschitz continuous with parameter $\rho + 2$.
\end{lemma}

\begin{proof} 
  We start by noting that $g_d$ in~\eqref{eq:squarepos} can be written
  as $g_d(u) = d^2_C(u)$ where $C$ is the closed ball with radius~$d$
  centered at the origin, and $d_C$ denotes the distance to the closed
  convex set~$C$. It is known that $g_d$ is convex, differentiable,
  the gradient is given by $\nabla g_d(u) = 2 (u - \pi(u) )$ and it is
  Lipschitz continuous with
  parameter~$2$~\cite[X.3.2.3]{UrrutyMarechal1993}. Also, function
  $h_d$ in~\eqref{eq:huber} is convex and differentiable. Thus, the
  function $\Psi_\omega$ is convex (resp. differentiable) as a sum of
  three convex (resp. differentiable) functions. It is strongly convex
  with parameter $\rho$ due to the first term $\frac{\rho}{2} \left\|
    \cdot - w \right\|^2$. The gradient in~\eqref{eq:ug} is
  clear. Finally, from $| \dot h_d(r) | - \dot h_d(s) | \leq 2 | r - s
  |$ for all $r, s$, there holds for any $u_1, u_2$, $$| \dot h_d(
  v^\top u_1 - d) - \dot h_d( v^\top u_2 - d ) | \leq 2 | v^\top (u_1
  - u_2) | \leq 2 \| u_1 - u_2 \|,$$ where $\| v \| = 1$ and the
  Cauchy-Schwarz inequality was used in the last step. We conclude
  from~\eqref{eq:ug} that, for any $u_1, u_2$, $$\| \nabla
  \Psi_\omega( u_1 ) - \nabla \Psi_\omega(u_2) \| \leq ( \rho + 2
  \omega + 2 (1-\omega) ) \left\| u_1 - u_2 \right\|,$$ \ie, the
  gradient of $\Psi_\omega$ is Lipschtz continuous with parameter~$\rho + 2$.
\end{proof}

Using Lemma~\ref{step3}, we see that the infimum of $\Psi_\omega$ is
attained at a single $u^\star(\omega)$ since it is a continuous,
strongly convex function. The derivative of $\psi$ in~\eqref{step2}
relies on $u^\star(\omega)$, as seen in Lemma~\ref{lem:psi-derivative}.

\begin{lemma}
  \label{lem:psi-derivative}
  Function $\psi$ in~\eqref{step2} is differentiable and its
  derivative is
  \begin{equation}
    \label{eq:psi-derivative}
    \dot \psi(\omega) = g_d\left( u^\star(\omega) \right)
  - h_d\left( v^\top u^\star(\omega) - d\right).
  \end{equation}
\end{lemma}
\begin{proof}
  We begin by bounding the norm of $u^\star(\omega)$. From the
  necessary stationary condition $\nabla \Psi_\omega( u^\star(\omega)
  ) = 0$ and~\eqref{eq:ug} we conclude
  \begin{equation}
    ( \rho + 2 \omega) u^\star(\omega) = \rho w + 2 \omega
    \pi( u^\star(\omega) ) - (1-\omega) \dot h_d(v^\top u^\star(\omega) -
    d) v. \label{step4}
  \end{equation} 
  Since $| \dot h_d(t) | \leq 2 d$ for all $t$ (see~\eqref{eq:huber}),
  $\left\| \pi(u) \right\| \leq d$ for all $u$, $\left\| v \right\| =
  1$, and $0 \leq \omega \leq 1$, we can bound the norm of the
  right-hand side of~\eqref{step4} by $\rho \left\| w \right\| + 4
  d$. Thus,
  \begin{eqnarray*}
    \left\| u^\star( \omega ) \right\| & \leq
    & \frac{1}{\rho + 2 \omega} (\rho \left\| w \right\| + 4 d) \\ 
    & \leq & \frac{1}{\rho} (\rho \left\| w \right\|+ 4 d) \\ 
    & = & \left\| w \right\| + \frac{4 d }{\rho}. 
  \end{eqnarray*}
  Introduce the compact set $U = \{ u \in {\mathbb R}^n\, : \, \left\|
    u \right\| \leq \|w\| + 4 d / \rho \}$. The previous analysis has
  shown that the dual function in~\eqref{step2} can also be
  represented as $\psi(\omega) = \inf\{ \Psi(\omega, u)\,:\, u \in U
  \}$, i.e., we can restrict the search to~$U$ and view $\Psi$ as
  defined in $\Omega \times U$.  We can thus invoke
  Lemma~\ref{lem:diff-of-Psii} to conclude that $\psi$ is
  differentiable and~\eqref{eq:psi-derivative} holds.
\end{proof}

\subsubsection*{Finding $u^\star(\omega)$} 
To obtain $u^\star(\omega)$ we must minimize $\Psi_\omega$. But, given
its properties in Lemma~\ref{step3}, the simple optimal Nesterov
method, described in Algorithm~\ref{alg:nesterov}, is also applicable
here.

\subsection{Solving problem~\eqref{eq:secondary-s2a-nodei}}
\label{sec:problems-nodes-in-A_i} Note that node~$i$ stores $x_{i}(t)$
and $\mu_{ik}(t)$, $k \in {\mathcal A}_i$. Thus, it can indeed address
Problem~\eqref{eq:secondary-s2a-nodei}.
Problem~\eqref{eq:secondary-s2a-nodei} is similar (in fact, much
simpler) than~\eqref{eq:secondary-s2s-nodei}, and following the
previous steps leads to the same Nesterov's optimal method. We omit
this straightforward derivation.

\section{ADMM: Solving Problem~\eqref{eq:admm-primary-vars}}
\label{sec:solv-probl-primary-vars}

Looking at~\eqref{eq:auglagrangian}, it is clear that
Problem~\eqref{eq:admm-primary-vars} decouples across nodes also.
Furthermore, at node~$i$ a simple unconstrained quadratic problem with
respect to $x_i$ must be solved, whose closed-form solution is

\begin{IEEEeqnarray}{rCL} 
  \isdraft{}{\nonumber}
  x_i (t+1) & = & \frac{1}{ |\overline{\mathcal V}_i | + | {\mathcal
      A}_i |} \left( \sum_{j \in \overline{\mathcal V}_i } \left(
      \frac{1}{\rho} \lambda_{ji}(t) + y_{ji}(t+1) \right) \isdraft{}{\right.  \\
  & & \left.} + \sum_{k \in {\mathcal A}_i} \left( \frac{1}{\rho}
      \mu_{ik}(t) + z_{ik}(t+1) \right) \right).
  \label{eq:primary-closed-form}
 \end{IEEEeqnarray}
 
 For node~$i$ to carry this update, it needs first to receive
$y_{ji}(t+1)$ from its neighbors~$j \in {\mathcal V}_i$. This requires
a communication step.

\section{ADMM: Implementing~\eqref{eq:admm-dual-update}
and~\eqref{eq:admm-dual-update2}}

Recall that the dual variable $\lambda_{ji}$ is maintained at both
nodes $i$ and $j$.  Node~$i$ can carry the update $\lambda_{ji}(t+1)$
in~\eqref{eq:admm-dual-update}, for all $j \in {\mathcal V}_i$, since
the needed data are available (recall that $y_{ji}(t+1)$ is
available from the previous communication step).  To update
$\lambda_{ij}(t+1) = \lambda_{ij}(t) + \rho ( y_{ij}(t+1) - x_j(t+1)$,
node~$i$ needs to receive $x_j(t+1)$ from its neighbors $j \in
{\mathcal V}_i$. This requires a communication step.

\section{Summary of the distributed algorithm}
\label{sec:summary-algorithm}

Our ADMM-based algorithm currently stops after a fixed number of
iterations, denoted $T$.
\begin{algorithm}
  \caption{Step~\ref{alg:mm-min} of DCOOL-NET: position updates}
  \label{alg:dsnl}
  \begin{algorithmic}[1] 
    \REQUIRE $x[l]$
    \ENSURE $x[l+1]$
    \FOR{$t=0$ \TO $T-1$}
      \FOR{each node $i \in \mathcal{V}$ in parallel} 
          \STATE Solve Problem~\eqref{eq:secondary-s2s-nodei} by
          minimizing $H$ in~\eqref{eq:secondary-s2s-Gamma} with Alg.~\ref{alg:nesterov} to obtain $y_{ij}(t+1)$, $j \in \overline{\mathcal V}_i$
          \FOR{$k=1$ \TO
            $|\mathcal{A}_i|$} \STATE Solve
          Problem~\eqref{eq:secondary-s2a-nodei} to obtain
          $z_{ik}(t+1)$ \label{alg:istep}
        \ENDFOR
        \STATE Send $y_{ij}(t+1)$ to neighbor $j \in
        \mathcal{V}_i$ \label{alg:send-secondary}
        \STATE Compute $x_i(t+1)$ from~\eqref{eq:primary-closed-form}
        \STATE Send $x_i(t+1)$ to all $j \in
        \mathcal{V}_i$ \label{alg:send-primary}
        \STATE Update $\{\lambda_{ji}(t+1),
         \mu_{ik}(t+1),\: j \in \overline{\mathcal{V}}_i, k \in {\mathcal
           A}_i \}$ as in~\eqref{eq:admm-dual-update}
         and~\eqref{eq:admm-dual-update2}
      \ENDFOR
    \ENDFOR
    \RETURN $x[l+1] = x(T)$
  \end{algorithmic}
\end{algorithm} 
Algorithm~\ref{alg:dsnl} outlines the procedure derived in
Secs.~\ref{sec:distr-algor},~\ref{sec:solv-probl-secondary-vars}
and~\ref{sec:solv-probl-primary-vars}, and corresponds to  step~\ref{alg:mm-min} of DCOOL-NET
(Algorithm~\ref{alg:mm}). Note that, in order to implement
step~\ref{alg:istep} of Algorithm~\ref{alg:dsnl}, one
must adapt Algorithm~\ref{alg:nesterov} to the problem at hand\footnote{In the spirit of reproducible research, all our MATLAB code will be made available online.}.

\subsection{Communication load}
\label{sec:communication-load} Algorithm~\ref{alg:dsnl} shows two
communication steps: step~\ref{alg:send-secondary} and
step~\ref{alg:send-primary}. At step~\ref{alg:send-secondary} each
node~$i$ sends $|\mathcal{V}_i|$ vectors in $\reals^p$, each to one
neighboring sensor, and at step~\ref{alg:send-primary} a vector in
$\reals^p$ is broadcast to all nodes in $\mathcal{V}_i$. This results
in $2 T L |\mathcal{V}_i|$ communications of $\reals^p$ vectors for
node~$i$ for the overall algorithm DCOOL-NET.  When comparing with SGO
in~\cite{ShiHeChenJiang2010}, for $T$ iterations, node~$i$ sends
$T|\mathcal{V}_{i}|$ vectors in $\reals^p$. The increase in
communications is the price to pay for the parallel nature of
DCOOL-NET.

\subsection{Initialization} We note that the tools introduced so
far could, in principle, also be used to set up a distributed
initialization scheme (\ie, to generate $x[0]$). More precisely, by
following the same steps leading to the ADMM formulation, a convex
function structured as~\eqref{eq:cvxapproximationFull} is amenable to
distributed optimization. For example, the ESOCP relaxation, used for
initialization in~\cite{ShiHeChenJiang2010} could fall into this
template. We leave a more in-depth exploration of this issue to future
work.

\section{Experimental results}
\label{sec:experimental-results}

\subsection{General setup}
\label{sec:experimental-setup}

Unless otherwise specified, the generated geometric networks are
composed by $4$ anchors and $50$ sensors, with an average node degree,
\ie, $\frac{1}{|\mathcal{V}|} \sum_{i \in \mathcal{V}}
|\mathcal{V}_{i}|$, of about $6$. In all experiments the sensors are
distributed at random and uniformly on a square of $1 \times 1$, and
anchors are placed, unless otherwise stated, at the four corners of
the unit square (to follow~\cite{ShiHeChenJiang2010}), namely, at
$(0,0)$,$(0,1)$, $(1,0)$ and $(1,1)$. These properties require a
communication range of about $R=0.24$.  Since localizability is an
issue when assessing the accuracy of sensor network localization
algorithms, the used networks are first checked to be generically
globally rigid, so that a small disturbance in measurements does not
create placement ambiguities. To detect generic global rigidity, we
used the methodologies
in\cite[Sec.~2]{AndersonShamesMaoFidan2010}. The results for the
proposed algorithm DCOOL-NET consider $L=40$ MM iterations, unless
otherwise stated.

\subsubsection*{Measurement noise}
\label{sec:measurement-noise} Throughout the experiments, the noise
model for range distance measurements is
\begin{equation}
  \label{eq:noise} 
  d_{ij} = \|x_{i}-x_{j}\|\cdot|n_{ij}|, \quad r_{ik} = \|x_{i} -
  a_{k}\|\cdot|n_{ik}|,
\end{equation}
where $n_{ij}, n_{ik} \sim \mathcal{N}(1,\sigma^{2})$ are independent
identically distributed Gaussian random variables with mean value $1$
and standard deviation $\sigma$ (specified ahead).  

This is the same model used in~\cite{ShiHeChenJiang2010}, against
which we compare our algorithm. Our choice of benchmark is motivated
by simulation results in~\cite{ShiHeChenJiang2010} showing that the
method is more accurate than the one proposed
in~\cite{SrirangarajanTewfikLuo2008}. As discussed in
Section~\ref{sec:problem-statement} and without loss of generality, we
assume $n_{ij}$ affects the edge measurement as a whole, \ie, both
$d_{ij}$ and $d_{ji}$ correspond to the same noisy measurement.

\subsubsection*{Initialization noise}
\label{sec:initialization-noise} 
To initialize the algorithms we take the true sensor
positions~$x^\star = \{ x_i^\star\, :\, i \in {\mathcal V} \}$ and we
perturb them by adding independent zero mean Gaussian noise, according
to
\begin{equation}
  \label{eq:init-noise} 
  x_{i}[0] = x_{i}^\star + \eta_{i},
\end{equation} 
where $\eta_{i} \sim \mathcal{N}(0,\sigma_{\mathrm{init}}^{2} I_p)$
and $I_p$ is the identity matrix of size $p\times p$. The parameter
$\sigma_{\mathrm{init}}$ is detailed ahead.

\subsubsection*{Accuracy measure}
\label{sec:accuracy-measure}

To quantify the precision of both algorithms we use Root Mean Square
Error (RMSE), defined as
\begin{equation}
  \label{eq:rmse} 
  \mathrm{RMSE} = \sqrt{\frac{1}{|\mathcal{V}| \mathrm{MC}}\sum_{m =
      1}^{\mathrm{MC}} \mathrm{SE}_m},
\end{equation} 
where 
\begin{equation} 
  \mathrm{SE}_m = \| \hat{x}_m - x^\star\|^{2} \label{eq:SEdef} 
\end{equation} 
denotes the network-wide squared error obtained at the $m$th Monte
Carlo trial, MC is the number of Monte Carlo trials and $\hat{x}_m$ is
the estimate of the sensors' positions at the $m$th Monte Carlo
trial. We emphasize that Equation~\eqref{eq:rmse} gives an accuracy
measure per sensor node in the network, and not the overall RMSE. This
is useful, not only to compare the performance between networks of
different sizes, but also to be able to extract a physical
interpretation of the quantity.


\subsection{Proposed and quadratic majorizers: RMSE vs. initialization
noise}
\label{sec:major-funct-comp} 
We compare the performance of our proposed majorizer
in~\eqref{eq:cvxapproximationFull} with a standard one built out of
quadratic functions, e.g., the one used
in~\cite{OguzGomesXavierOliveira2011}.  We have submitted a simple
source localization problem with one sensor and $4$ anchors to two MM
algorithms, each associated with one of the majorization
functions. They ran for a fixed number of $30$ iterations. At each
Monte Carlo trial, the true sensor positions were corrupted by zero mean
Gaussian noise, as in~\eqref{eq:init-noise}, with standard deviation
$\sigma_{\mathrm{init}} \in [0.01,1]$. The range measurements are
taken to be noiseless, \ie, $\sigma = 0$ in~\eqref{eq:noise}, in order
to create an idealized scenario for direct comparison of the two
approaches.
\begin{figure}[!t] 
  \centering
  \includegraphics[width=0.49\textwidth]{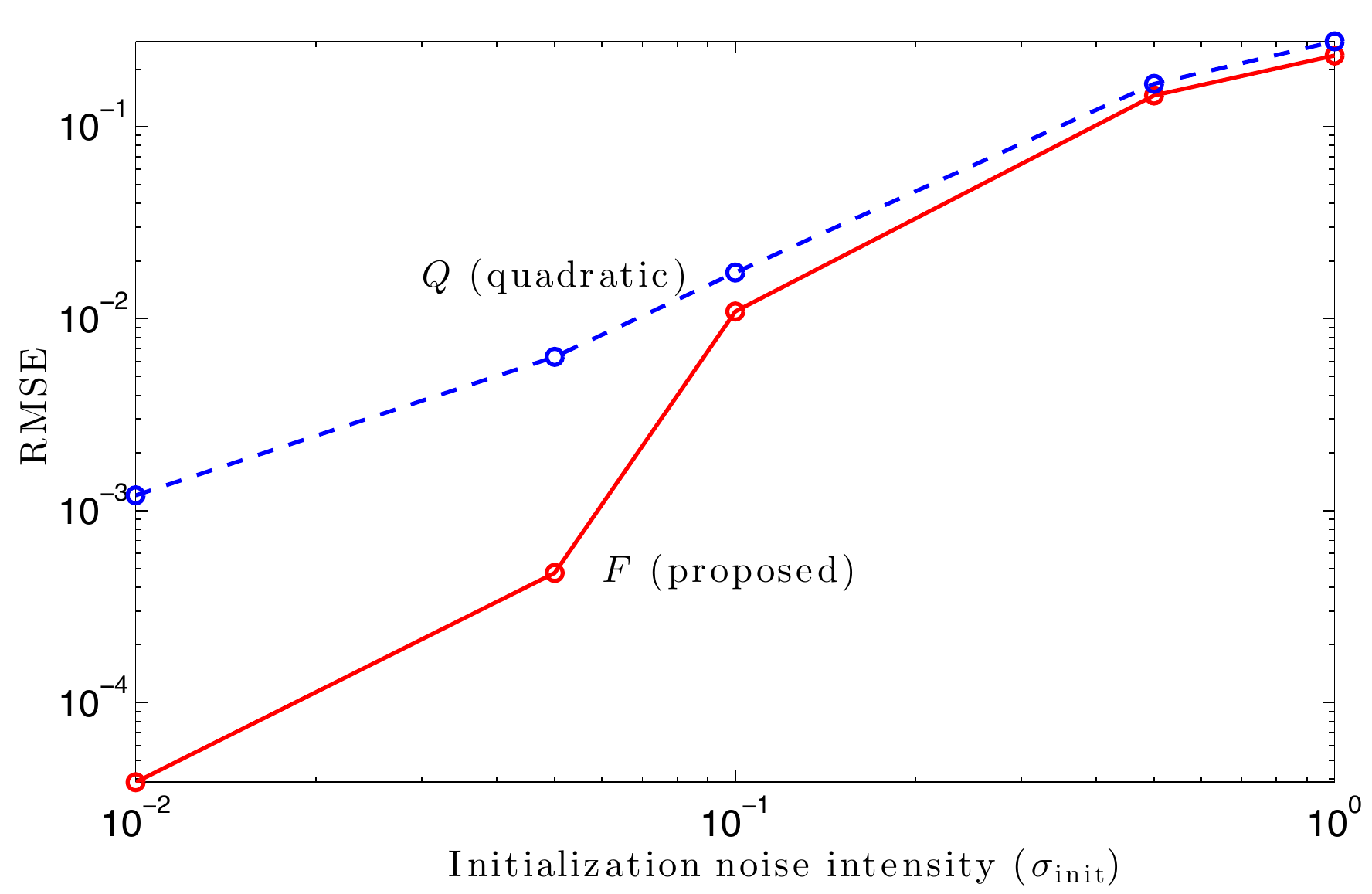}
  \caption{RMSE vs. $\sigma_{\mathrm{init}}$, the intensity of
    initialization noise in~\eqref{eq:init-noise}. The range
    measurements are noiseless: $\sigma = 0$
    in~\eqref{eq:noise}. Anchors are at the unit square corners. The
    proposed majorizer (red, solid) outperforms the quadratic
    majorizer (blue, dashed) in accuracy.}
  \label{fig:majcompinit}
\end{figure} 
The evolution of RMSE as a function of initialization noise
intensity is illustrated in Fig.~\ref{fig:majcompinit}. There is a
clear advantage of using this majorization function when the
initialization is within a radius of the true location which is $30\%$
of the square size.

\subsection{DCOOL-NET and SGO: RMSE vs. initialization noise}
\label{sec:depend-accur-with-init}

Two sets of experiments were made to compare the RMSE performance of
SGO in~\cite{ShiHeChenJiang2010} and the proposed DCOOL-NET, as a
function of the initialization quality (\ie, $\sigma_{\mathrm{init}}$
in~\eqref{eq:init-noise}).  In the first set, range measurements are
noiseless (\ie, $\sigma = 0$ in~\eqref{eq:noise}), whereas in the second
set we consider noisy range measurements ($\sigma > 0$).

\subsubsection{Noiseless range measurements}
\label{sec:sens-anch-cvx-hull-no-noise}

In this setup $300$ Monte Carlo trials were run. As the measurements
are accurate ($\sigma = 0$ in~\eqref{eq:noise}) one would expect not
only insignificant values of RMSE, but also a considerable agreement
between all the Monte Carlo trials on the solution for sufficiently
close initializations.
\begin{figure}[!t] 
  \centering
  \includegraphics[width=0.49\textwidth]{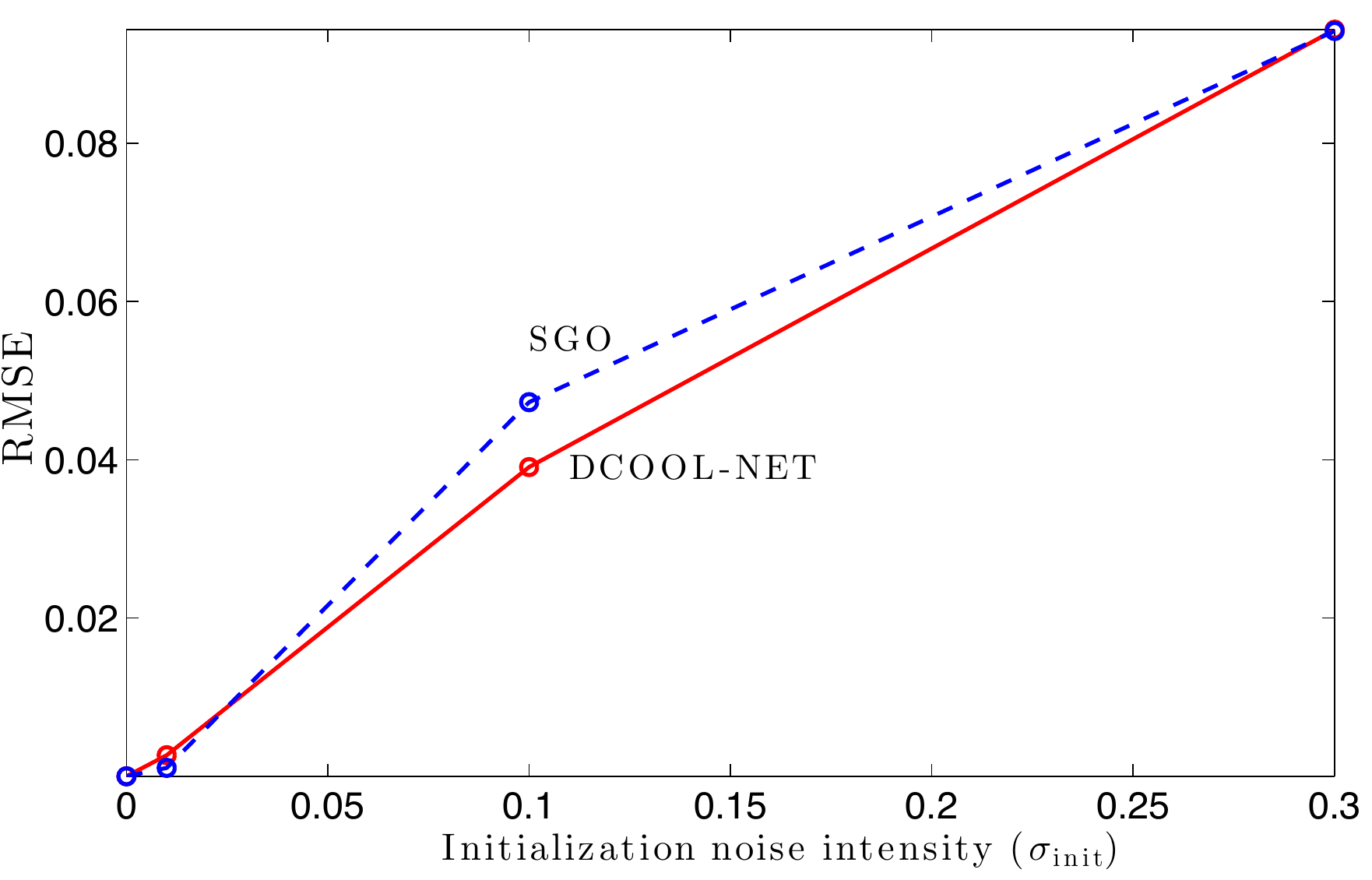}
  \caption{RMSE vs. $\sigma_{\mathrm{init}}$, the intensity of
    initialization noise in~\eqref{eq:init-noise}. The range measurements
    are noiseless: $\sigma = 0$ in~\eqref{eq:noise}. Anchors are at the
    unit square corners. Proposed DCOOL-NET (red, solid) and SGO (blue,
    dashed) attain comparable accuracy.}
  \label{fig:e1}
\end{figure} 
Fig.~\ref{fig:e1} confirms that both DCOOL-NET and SGO
achieve small error positions, and their accuracies are comparable. As
stated before, SGO also has a low computational complexity. In fact, lower than DCOOL-NET (although DCOOL-NET is fully parallel across nodes, whereas SGO operates by activating the nodes sequentially, implying some high-level coordination).
\begin{table}[!t] 
  \centering
  \caption{Squared error dispersion over Monte Carlo trials for Fig.~
    \ref{fig:e1}.}
  \label{tab:e1} 
  \begin{tabular}{@{}lclcl@{}}
    \toprule
    \textbf{$\sigma_\mathrm{init}$}&\textbf{DCOOL-NET}&\textbf{SGO}\\\midrule
    0.01&0.0002&0.0007\\
    0.10&0.0638&0.1290\\
    0.30&0.2380&0.3400\\
    \bottomrule
  \end{tabular}
\end{table} 
Tab.~\ref{tab:e1} shows the squared error dispersion over all Monte
Carlo trials, \ie, the standard deviation of the data $\{ \mathrm{SE}_m\,:\, m = 1, \ldots, \mathrm{MC} \}$, recall~\eqref{eq:SEdef}, for both algorithms. We see that
DCOOL-NET exhibits a more stable performance, in the sense that it has
a lower squared error dispersion.


\subsubsection{Noisy range measurements}
\label{sec:sensors-with-noisy}

We set $\sigma = 0.12$ in the noise model~\eqref{eq:noise}.  
\begin{figure}[!t] 
  \centering
  \includegraphics[width=0.49\textwidth]{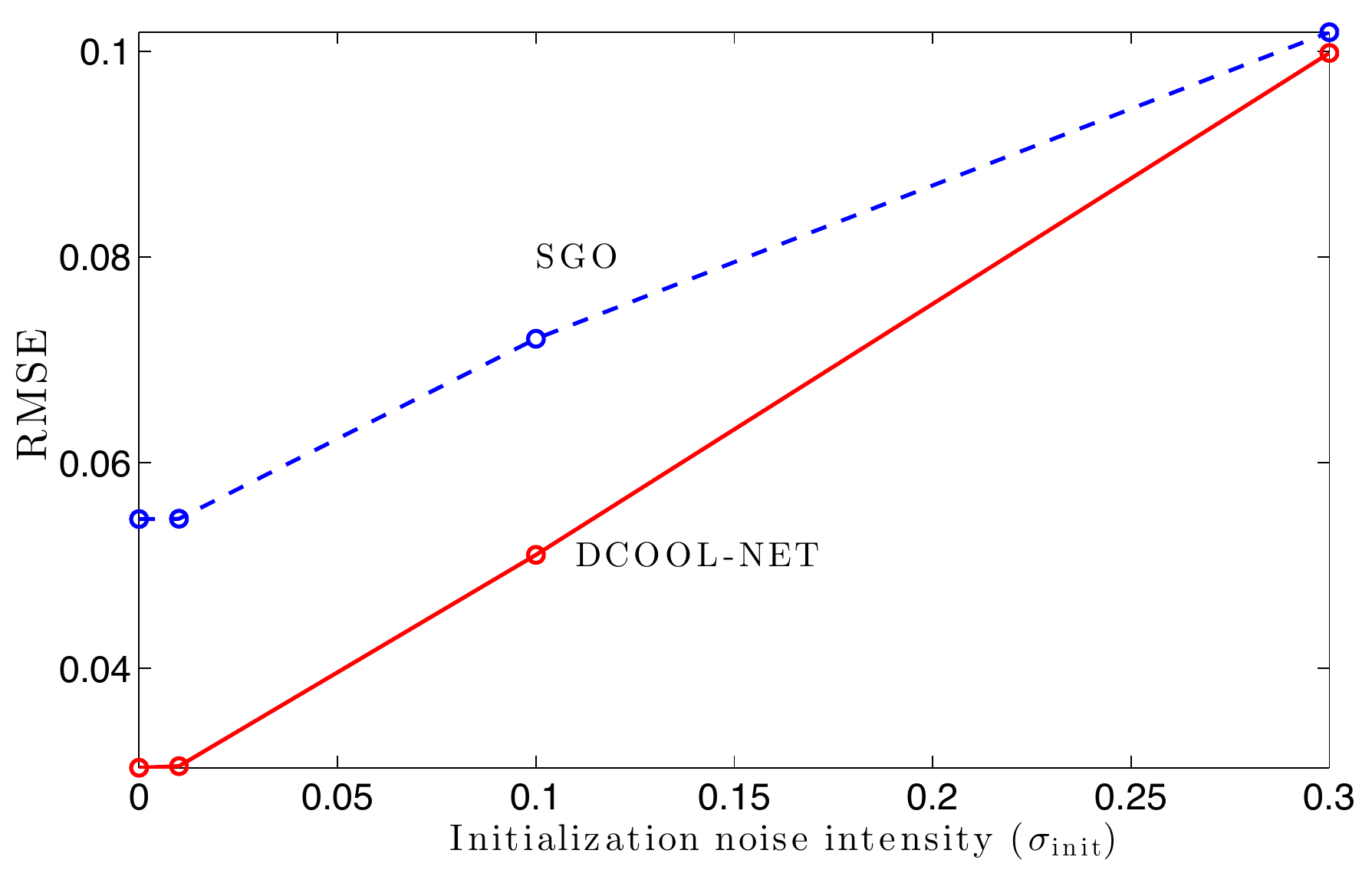}
  \caption{RMSE vs. $\sigma_{\mathrm{init}}$, the intensity of
    initialization noise in~\eqref{eq:init-noise}. The range
    measurements are noisy: $\sigma = 0.12$
    in~\eqref{eq:noise}. Anchors are at the unit square
    corners. Proposed DCOOL-NET (red, solid) outperforms SGO (blue,
    dashed) in accuracy.}
  \label{fig:e8}
\end{figure}
Fig.~\ref{fig:e8} shows that DCOOL-NET fares better than SGO: the gap
between the performances of both algorithms is now quite significant.
\begin{table}[!t] 
  \centering
  \caption{Squared error dispersion over Monte Carlo trials for
    Fig.~\ref{fig:e8}.}
  \label{tab:e8} 
  \begin{tabular}{@{}lclcl@{}}
    \toprule
    \textbf{$\sigma_\mathrm{init}$}&\textbf{DCOOL-NET}&\textbf{SGO}\\\midrule
    0.00&0.0118&0.0783\\
    0.01&0.0121&0.0775\\
    0.10&0.0727&0.1610\\
    0.30&0.2490&0.3320\\
    \bottomrule
  \end{tabular}
\end{table} 
The squared error dispersion over all Monte Carlo trials for both
algorithms is given in Tab.~\ref{tab:e8}. As before, we see that
DCOOL-NET is more reliable, in the sense that it exhibits lower variance of
estimates across Monte Carlo experiments.


\begin{figure}[!t] 
  \centering
  \includegraphics[width=0.49\textwidth]{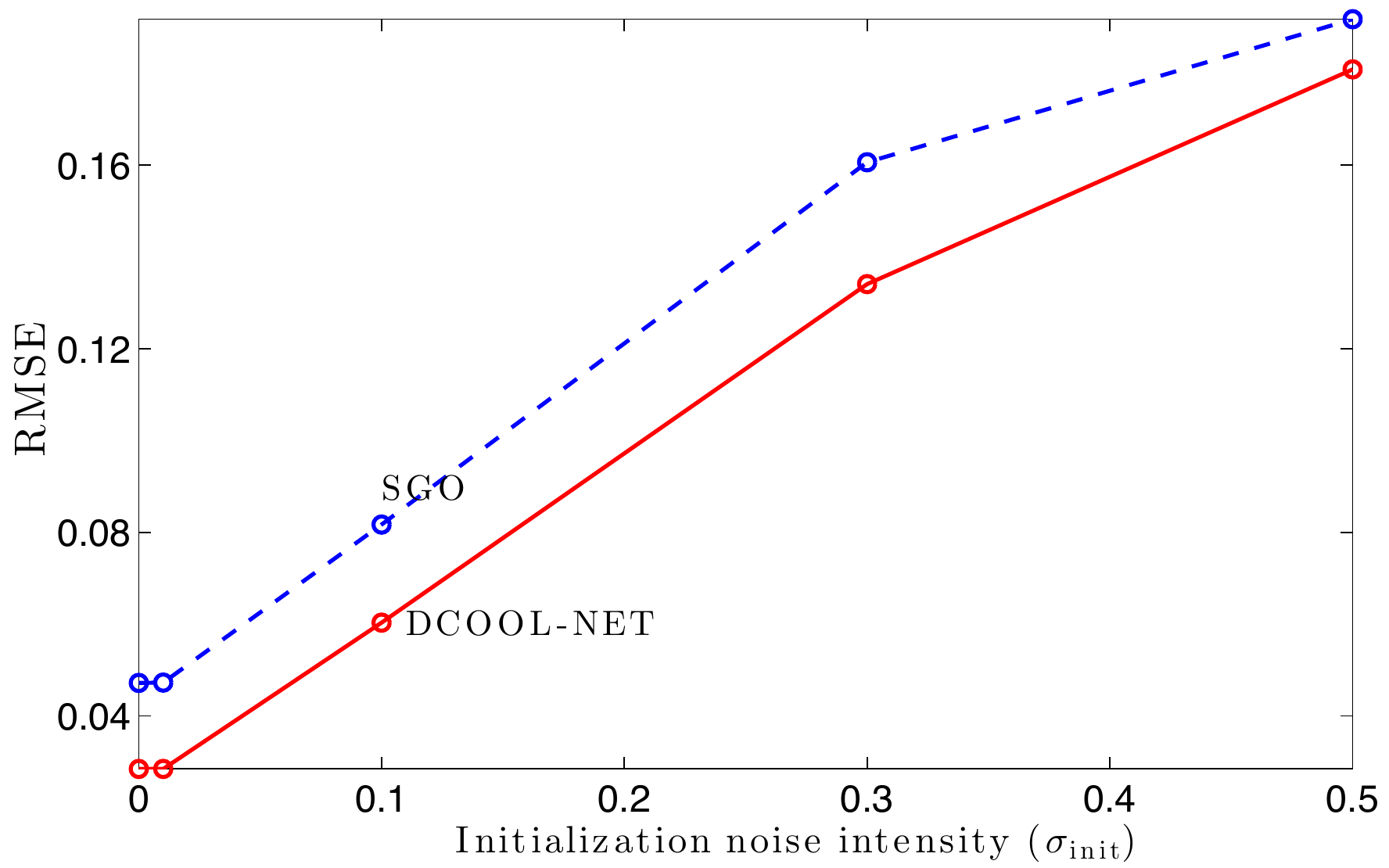}
  \caption{RMSE vs. $\sigma_{\mathrm{init}}$, the intensity of
    initialization noise in~\eqref{eq:init-noise}. The range measurements
    are noisy: $\sigma = 0.12$ in~\eqref{eq:noise}. Anchors were randomly
    placed in the unit square. Proposed DCOOL-NET (red, solid) outperforms
    SGO (blue, dashed) in accuracy.}
  \label{fig:e7}
\end{figure}

We also considered placing the anchors randomly within the unit
square, instead of at the corners. This is a more realistic and
challenging setup, where the sensors are no longer necessarily located
inside the convex hull of the anchors. The corresponding results are
shown in Fig.~\ref{fig:e7} and Tab.~\ref{tab:e7}, for $250$ Monte
Carlo trials.
\begin{table}[!t] 
  \centering
  \caption{Squared error dispersion over Monte Carlo trials for Fig.~
    \ref{fig:e7}.}
  \label{tab:e7} 
\begin{tabular}{@{}lclcl@{}}
\toprule
\textbf{$\sigma_\mathrm{init}$}&\textbf{DCOOL-NET}&\textbf{SGO}\\\midrule
0.00&0.0097&0.0712\\
0.01&0.0099&0.0709\\
0.10&0.1550&0.3160\\
0.30&0.4350&0.8440\\
0.50&0.8330&1.3000\\
\bottomrule
\end{tabular}
\end{table} 
Again, DCOOL-NET achieves better accuracy. Comparing the dispersions
in Tabs.~\ref{tab:e8} and~\ref{tab:e7} also reveals that the gap in
reliability between SGO and our algorithm is now wider.

\subsection{DCOOL-NET and SGO: RMSE vs. measurement noise}
\label{sec:depend-accur-with-meas}

To evaluate the sensitivity of both algorithms to the intensity of
noise present in range measurements (\ie, $\sigma$
in~\eqref{eq:noise}), $300$ Monte Carlo trials were run for $\sigma =
0.01, 0.1, 0.12, 0.15, 0.17, 0.2, 0.3$. Both algorithms were
initialized at the true sensor positions, \ie, $\sigma_{\mathrm{init}}
= 0$ in~\eqref{eq:init-noise}, and DCOOL-NET performs $L = 100$
iterations\footnote{This is to guarantee that, in practice, DCOOL-NET
  indeed attained a fixed point, but the results barely changed for $L
  = 40$.}.
\begin{figure}[!t] 
  \centering
  \includegraphics[width=0.49\textwidth]{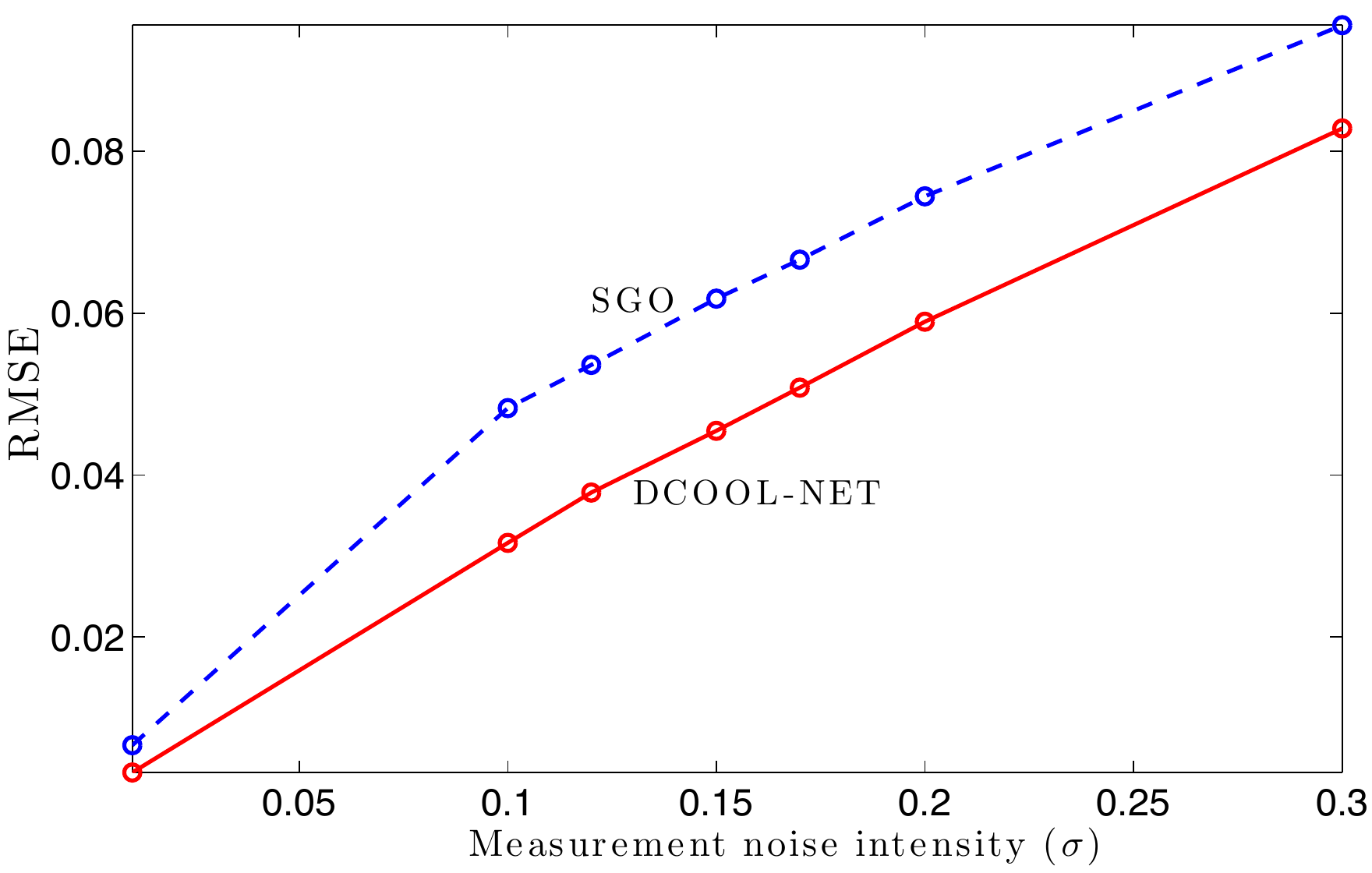}
  \caption{RMSE vs. $\sigma$, the intensity of measurement noise
    in~\eqref{eq:noise}. No initialization noise:
    $\sigma_{\mathrm{init}} = 0$ in~\eqref{eq:init-noise}. Anchors are
    at the unit square corners. Proposed DCOOL-NET (red, solid)
    outperforms SGO (blue, dashed) in accuracy.}
  \label{fig:e2}
\end{figure} 
Fig.~\ref{fig:e2} and Tab.~\ref{tab:e2} summarize the
computer simulations for this setup. As before, DCOOL-NET consistently
achieves better accuracy and stability.  
\begin{table}[!t] 
  \centering
  \caption{Squared error dispersion over Monte Carlo trials for Fig.~
    \ref{fig:e2}.}
  \label{tab:e2} 
\begin{tabular}{@{}lclcl@{}}
\toprule
\textbf{$\sigma$}&\textbf{DCOOL-NET}&\textbf{SGO}\\\midrule
0.01&0.0002&0.0016\\
0.10&0.0177&0.0688\\
0.12&0.0218&0.0702\\
0.15&0.0326&0.0921\\
0.17&0.0394&0.0993\\
0.20&0.0525&0.1090\\
0.30&0.1020&0.1630\\
\bottomrule
\end{tabular}
\end{table} 

\subsection{DCOOL-NET and SGO: RMSE vs. communication cost}
\label{sec:accur-numb-comm}

We assessed how the RMSE varies with the communication load incurred
by both algorithms.  We considered the general setup described in
Sec.~\ref{sec:experimental-setup}.  The results are displayed in
Fig.~\ref{fig:comm}. We see an interesting tradeoff: SGO converges
much quicker than DCOOL-NET (in terms of communication rounds), and
attains a lower RMSE sooner. However, DCOOL-NET can improve its
accuracy through more communications, whereas SGO remains trapped in a
suboptimal solution.

\begin{figure}[!t] 
  \centering
  \includegraphics[width=0.49\textwidth]{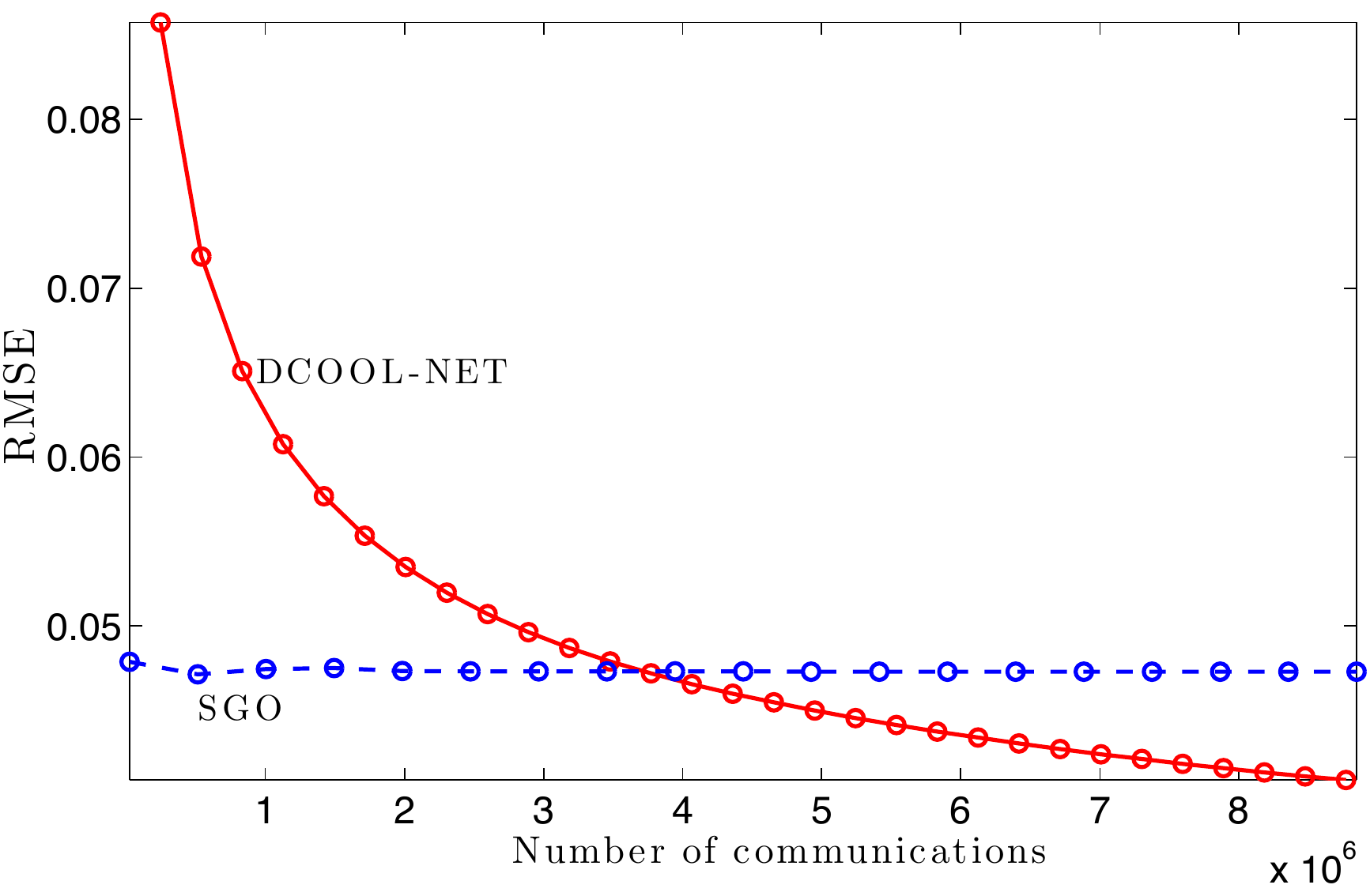}
  \caption{RMSE versus total number of two-dimensional vectors
    communicated in the network. The range measurements are noiseless:
    $\sigma = 0$ in~\eqref{eq:noise}. Initialization is noisy:
    $\sigma_{\mathrm{init}} = 0.1$ in~\eqref{eq:init-noise}. Anchors are
    at the unit square corners. Proposed DCOOL-NET (red, solid)
    outperforms SGO (blue, dashed) in accuracy, at the expense of more
    communications.}
  \label{fig:comm}
\end{figure} 

\subsection{DCOOL-NET: RMSE vs. parameter $\rho$}
\label{sec:sens-param-rho}

The parameter $\rho$ plays a role in the augmented Lagrangian
discussed in Sec.~\ref{sec:distr-algor}, and is user-selected.  As
such, it is important to study the sensitivity of DCOOL-NET to this
parameter choice.  For this purpose, we have tested several $\rho$
between $1$ and $200$. For each choice, $300$ Monte Carlo trials were
performed using noisy measurements and initializations.
\begin{figure}[!t] 
  \centering
  \includegraphics[width=0.49\textwidth]{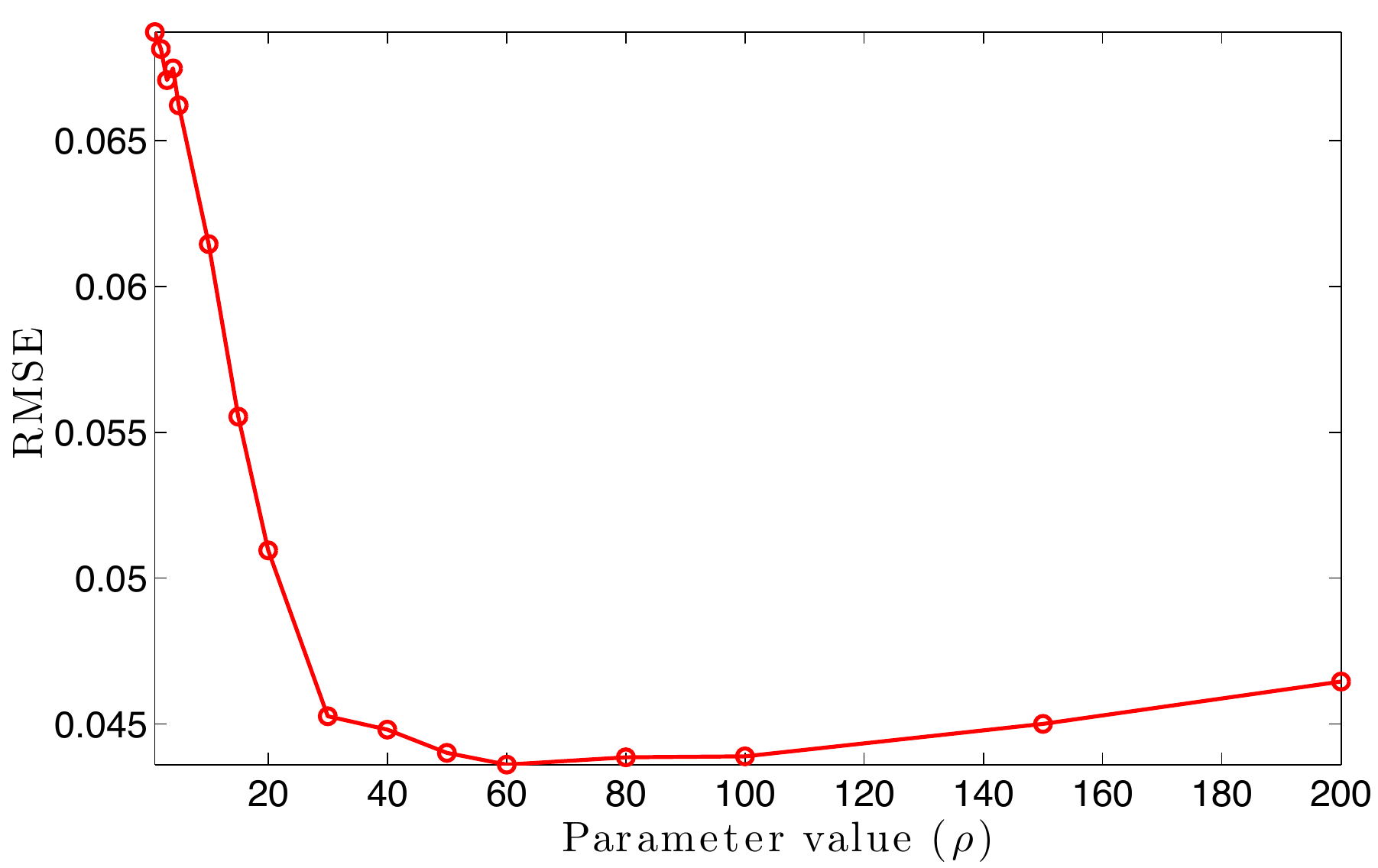}
  \caption{RMSE vs. $\rho$. The range measurements are noisy: $\sigma
    = 0.05$ in~\eqref{eq:noise}. Initialization is noisy:
    $\sigma_{\mathrm{init}}=0.1$ in~\eqref{eq:init-noise}. Anchors are the
    unit square corners. }
  \label{fig:rmsevsrho}
\end{figure} 
Fig.~\ref{fig:rmsevsrho} portrays RMSE against~$\rho$ for $L = 40$
iterations of DCOOL-NET.  There is no ample variation, especially for
values of $\rho$ over $30$, which offers some confidence in the
algorithm resilience to this parameter, a pivotal feature from the
practical standpoint. However, an analytical approach for selecting
the optimal~$\rho$ is beyond the scope of this work, and is postponed
for future research. Note that adaptive schemes to adjust~$\rho$ do
exist for centralized
settings, \eg,~\cite{BoydParikhChuPeleatoEckstein2011}, but seem impractical
for distributed setups as they require global computations.

\section{Conclusions}
\label{sec:conclusions}

Sensor network localization based on noisy range measurement between
some pairs of sensors is a current topic of great interest, which maps
into a difficult nonconvex optimization problem, once cast in a
maximum likelihood (ML) framework. Efficient algorithms guaranteed to
find the global minimum are not known even for the centralized
setting.  In this work, we proposed an algorithm, termed DCOOL-NET,
for the more challenging distributed setup in which no central or
fusion node is available: instead, neighbor nodes collaboratively
exchange messages in order to solve the underlying optimization
problem. DCOOL-NET stems from a majorization-minorization (MM)
approach and capitalizes on a novel convex majorizer. The proposed
majorizer is a main contribution of this work: it is tuned to the
nonconvexities of the cost function, which translates into better
performance when directly compared with traditional MM quadratic
majorizers. More importantly, the new majorizer exhibits several
important properties for the problem at hand: it allows for
distributed, parallel, optimization via the alternating direction
method of multipliers (ADMM) and for low-complexity solvers based on
fast-gradient Nesterov methods to tackle the ADMM subproblems at each
node. DCOOL-NET decreases the cost function at each iteration, a
property inherited from the MM framework, but it is not guaranteed to
find the global minimum (a theoretical limitation shared by all
existing distributed and non-distributed algorithms). However,
computer simulations show that DCOOL-NET achieve better sensor
positioning accuracy than a state-of-art method distributed algorithm
which, furthermore, is not parallel.

\appendices
\section{Proof of Proposition~\ref{prop:majorization-properties}}
\label{sec:proof-theor-thcvxmaj}

We write $\Phi_d(u)$ instead of $\Phi_d( u | v )$ and we let $\langle
x, y \rangle = x^\top y$.

\subsubsection*{Convexity} Note that $g_d$ is convex as the
composition of the convex, non-decreasing function $( \cdot )_+^2$
with the convex function $\left\| \cdot \right\| - d$. Also, $h_d(
\langle v / \| v \| , \cdot \rangle - d)$ is convex as the composition
of the convex Huber function $h_d( \cdot)$ with the affine map
$\langle v/ \| v \|, \cdot \rangle - d$. Finally, $\Phi_d$ is convex
as the pointwise maximum of two convex functions.

\subsubsection*{Tightness} It is straightforward to check that
$\phi_d(v) = \Phi_d(v)$ by examining separately the three cases $\|
v\| < d$, $d \leq \| v \| < 2d$ and $\| v \| \geq 2d$.

\subsubsection*{Majorization} We must show that $\Phi_d(u) \geq
\phi_d(u)$ for all $u$.  First, consider $\| u \| \geq d$. Then,
$g_d(u) = \phi_d(u)$ and it follows that $\Phi_d(v) = \max\{ g_d(u),
h_d( \langle v / \| v \|, u \rangle - d) \} \geq \phi_d(u)$.  Now,
consider $\| u \| < d$ and write $u = R \hat u$, where $R = \| u \| <
d$ and $\| \hat u \| = 1$. It is straightforward to check that, in
terms of $R$ and $\hat u$, we have $\phi_d(u) = ( R - d )^2$ and
$\Phi_d( u ) = h_d( R \langle \hat v, \hat u \rangle - d)$, where
$\hat v = v / \| v \|$. Thus, we must show that $h_d( R \langle \hat
v, \hat u \rangle - d) \geq (R - d)^2$. Motivated by the definition of
the Huber function~$h_d$ in two branches, we divide the analysis in
two cases.

Case 1: $| R \langle \hat v, \hat u \rangle - d | \leq d$. In this
case, $h_d( R \langle \hat v, \hat u \rangle - d) = ( R \langle \hat
v, \hat u \rangle - d)^2$. Noting that $| \langle \hat v, \hat u
\rangle | \leq 1$, there holds $$( R \langle \hat v, \hat u \rangle -
d)^2 \geq \inf\{ ( R z - d )^2\, : \, | z | \leq 1 \} = (R - d)^2,$$
where the fact that $R < d$ was used to compute the infimum over~$z$
(attained at $z = 1$).

Case 2: $| R \langle \hat v, \hat u \rangle - d | > d$. In this case,
$h_d( R \langle \hat v, \hat u \rangle - d) = 2 d | R \langle \hat v,
\hat u \rangle - d | - d^2$.  Thus, $$h_d( R \langle \hat v, \hat u
\rangle - d ) \geq d^2 \geq (d - R)^2,$$ where the last inequality
follows from $0 \leq R < d$.

\section{Proof of~\eqref{eq:auglagrangian-v2}}
\label{sec:proof-B}

We show how to rewrite~\eqref{eq:auglagrangian} as~\eqref{eq:auglagrangian-v2}.
First, note that $F(y, z)$ in~\eqref{eq:admmcost} can be rewritten as 
\begin{equation}
F(y, z) = \sum_i \sum_{j \in {\mathcal V}_i} F_{ij}( y_{ii}, y_{ij} ) + 2 \sum_i \sum_{k \in {\mathcal A}_i} F_{ik}( z_{ik} ).
\label{apB1}
\end{equation}
Here, we used the fact that $F_{ij}( y_{ji}, y_{jj} ) = F_{ji} ( y_{jj}, y_{ji} )$ which follows from $d_{ij} = d_{ji}$ and $\Phi_d( u | v ) = \Phi_d( -u | -v )$, see~\eqref{eq:majorization-simple}.
In addition, there holds 
\begin{IEEEeqnarray}{rCL} 
 & & \sum_i \sum_{j \in \overline{\mathcal V}_i} \lambda_{ji}^\top ( y_{ji} - x_i ) + \frac{\rho}{2} \left\| y_{ji} - x_i \right\|^2  \nonumber \\ & = & 
\sum_j \sum_{i \in \overline{\mathcal V}_j} \lambda_{ij}^\top ( y_{ij} - x_j ) + \frac{\rho}{2} \left\| y_{ij} - x_j \right\|^2 \nonumber \\ & = & 
\sum_i \sum_{j \in \overline{\mathcal V}_i} \lambda_{ij}^\top ( y_{ij} - x_j ) + \frac{\rho}{2} \left\| y_{ij} - x_j \right\|^2. \label{apB2}
 \end{IEEEeqnarray}
The first equality follows from interchanging $i$ with~$j$. The second equality follows from noting that $i \in \overline{\mathcal V}_j$ if and only if $j \in \overline{\mathcal V}_i$.
Using~\eqref{apB1} and~\eqref{apB2} in~\eqref{eq:auglagrangian} gives~\eqref{eq:auglagrangian-v2}.

\bibliographystyle{IEEEtran} \bibliography{IEEEabrv,biblos.bib}

\begin{thebibliography}{10}
\providecommand{\url}[1]{#1}
\csname url@samestyle\endcsname
\providecommand{\newblock}{\relax}
\providecommand{\bibinfo}[2]{#2}
\providecommand{\BIBentrySTDinterwordspacing}{\spaceskip=0pt\relax}
\providecommand{\BIBentryALTinterwordstretchfactor}{4}
\providecommand{\BIBentryALTinterwordspacing}{\spaceskip=\fontdimen2\font plus
\BIBentryALTinterwordstretchfactor\fontdimen3\font minus
  \fontdimen4\font\relax}
\providecommand{\BIBforeignlanguage}[2]{{%
\expandafter\ifx\csname l@#1\endcsname\relax
\typeout{** WARNING: IEEEtran.bst: No hyphenation pattern has been}%
\typeout{** loaded for the language `#1'. Using the pattern for}%
\typeout{** the default language instead.}%
\else
\language=\csname l@#1\endcsname
\fi
#2}}
\providecommand{\BIBdecl}{\relax}
\BIBdecl

\bibitem{lai2011}
D.~Lai, R.~Begg, and M.~Palaniswami, \emph{Healthcare Sensor Networks:
  Challenges Toward Practical Implementation}.\hskip 1em plus 0.5em minus
  0.4em\relax Taylor \& Francis, 2011.

\bibitem{KellerGur2011}
Y.~Keller and Y.~Gur, ``A diffusion approach to network localization,''
  \emph{Signal Processing, IEEE Transactions on}, vol.~59, no.~6, pp. 2642
  --2654, jun. 2011.

\bibitem{DestinoAbreu2011}
G.~Destino and G.~Abreu, ``On the maximum likelihood approach for source and
  network localization,'' \emph{Signal Processing, IEEE Transactions on},
  vol.~59, no.~10, pp. 4954 --4970, oct. 2011.

\bibitem{OguzGomesXavierOliveira2011}
P.~Oguz-Ekim, J.~Gomes, J.~Xavier, and P.~Oliveira, ``Robust localization of
  nodes and time-recursive tracking in sensor networks using noisy range
  measurements,'' \emph{Signal Processing, IEEE Transactions on}, vol.~59,
  no.~8, pp. 3930 --3942, aug. 2011.

\bibitem{BiswasLiangTohYeWang2006}
P.~Biswas, T.-C. Liang, K.-C. Toh, Y.~Ye, and T.-C. Wang, ``Semidefinite
  programming approaches for sensor network localization with noisy distance
  measurements,'' \emph{Automation Science and Engineering, IEEE Transactions
  on}, vol.~3, no.~4, pp. 360 --371, oct. 2006.

\bibitem{KorkmazVeen2009}
S.~Korkmaz and A.-J. van~der Veen, ``Robust localization in sensor networks
  with iterative majorization techniques,'' in \emph{Acoustics, Speech and
  Signal Processing, 2009. ICASSP 2009. IEEE International Conference on}, apr.
  2009, pp. 2049 --2052.

\bibitem{ShangRumiZhangFromherz2004}
Y.~Shang, W.~Rumi, Y.~Zhang, and M.~Fromherz, ``Localization from connectivity
  in sensor networks,'' \emph{Parallel and Distributed Systems, IEEE
  Transactions on}, vol.~15, no.~11, pp. 961 -- 974, nov. 2004.

\bibitem{CostaPatwariHero2006}
J.~Costa, N.~Patwari, and A.~Hero~III, ``Distributed weighted-multidimensional
  scaling for node localization in sensor networks,'' \emph{ACM Transactions on
  Sensor Networks (TOSN)}, vol.~2, no.~1, pp. 39--64, 2006.

\bibitem{SrirangarajanTewfikLuo2008}
S.~Srirangarajan, A.~Tewfik, and Z.-Q. Luo, ``Distributed sensor network
  localization using {SOCP} relaxation,'' \emph{Wireless Communications, IEEE
  Transactions on}, vol.~7, no.~12, pp. 4886 --4895, dec. 2008.

\bibitem{ChanSo2009}
F.~Chan and H.~So, ``Accurate distributed range-based positioning algorithm for
  wireless sensor networks,'' \emph{Signal Processing, {IEEE} Transactions on},
  vol.~57, no.~10, pp. 4100 --4105, oct. 2009.

\bibitem{KhanKarMoura2010}
U.~Khan, S.~Kar, and J.~Moura, ``{DILAND}: An algorithm for distributed sensor
  localization with noisy distance measurements,'' \emph{Signal Processing,
  IEEE Transactions on}, vol.~58, no.~3, pp. 1940 --1947, mar. 2010.

\bibitem{ShiHeChenJiang2010}
Q.~Shi, C.~He, H.~Chen, and L.~Jiang, ``Distributed wireless sensor network
  localization via sequential greedy optimization algorithm,'' \emph{Signal
  Processing, IEEE Transactions on}, vol.~58, no.~6, pp. 3328 --3340, jun.
  2010.

\bibitem{HunterLange2004}
D.~R. Hunter and K.~Lange, ``A tutorial on {MM} algorithms,'' \emph{The
  American Statistician}, vol.~58, no.~1, pp. 30--37, feb. 2004.

\bibitem{AndersonShamesMaoFidan2010}
B.~D.~O. Anderson, I.~Shames, G.~Mao, and B.~Fidan, ``Formal theory of noisy
  sensor network localization,'' \emph{{SIAM} Journal on Discrete Mathematics},
  vol.~24, no.~2, pp. 684--698, 2010.

\bibitem{ForeroGiannakis2012}
P.~Forero and G.~Giannakis, ``Sparsity-exploiting robust multidimensional
  scaling,'' \emph{Signal Processing, IEEE Transactions on}, vol.~60, no.~8,
  pp. 4118 --4134, aug. 2012.

\bibitem{BoydParikhChuPeleatoEckstein2011}
S.~Boyd, N.~Parikh, E.~Chu, B.~Peleato, and J.~Eckstein, ``Distributed
  optimization and statistical learning via the alternating direction method of
  multipliers,'' \emph{Foundations and Trends{\textregistered} in Machine
  Learning}, vol.~3, no.~1, pp. 1--122, 2011.

\bibitem{SchizasRibeiroGiannakis2008}
I.~Schizas, A.~Ribeiro, and G.~Giannakis, ``Consensus in ad hoc {WSN}s with
  noisy links --- part i: Distributed estimation of deterministic signals,''
  \emph{Signal Processing, IEEE Transactions on}, vol.~56, no.~1, pp. 350
  --364, jan. 2008.

\bibitem{ZhuGiannakisCano2009}
H.~Zhu, G.~Giannakis, and A.~Cano, ``Distributed in-network channel decoding,''
  \emph{Signal Processing, IEEE Transactions on}, vol.~57, no.~10, pp. 3970
  --3983, oct. 2009.

\bibitem{ForeroCanoGiannakis2010}
P.~Forero, A.~Cano, and G.~Giannakis, ``Consensus-based distributed support
  vector machines,'' \emph{The Journal of Machine Learning Research}, vol.~11,
  pp. 1663--1707, 2010.

\bibitem{BazerqueGiannakis2010}
J.~Bazerque and G.~Giannakis, ``Distributed spectrum sensing for cognitive
  radio networks by exploiting sparsity,'' \emph{Signal Processing, IEEE
  Transactions on}, vol.~58, no.~3, pp. 1847 --1862, mar. 2010.

\bibitem{ErsegheZennaroAneseVangelista2011}
T.~Erseghe, D.~Zennaro, E.~Dall'Anese, and L.~Vangelista, ``Fast consensus by
  the alternating direction multipliers method,'' \emph{Signal Processing, IEEE
  Transactions on}, vol.~59, no.~11, pp. 5523 --5537, nov. 2011.

\bibitem{MotaXavierAguiarPuschel}
J.~Mota, J.~Xavier, P.~Aguiar, and M.~Puschel, ``Distributed basis pursuit,''
  \emph{Signal Processing, IEEE Transactions on}, vol.~60, no.~4, pp. 1942
  --1956, apr. 2012.

\bibitem{UrrutyMarechal1993}
J.-B. Hiriart-Urruty and C.~Lemar{\'e}chal, \emph{Convex analysis and
  minimization algorithms}.\hskip 1em plus 0.5em minus 0.4em\relax
  Springer-Verlag Limited, 1993.

\bibitem{Nesterov2004}
Y.~Nesterov, \emph{Introductory Lectures on Convex Optimization: A Basic
  Course}.\hskip 1em plus 0.5em minus 0.4em\relax Kluwer Academic Publishers,
  2004.

\end{thebibliography}

\end{document}